\documentclass[reqno]{amsart}
\pdfoutput=1

\usepackage{amssymb}
\usepackage{amsmath}
\usepackage{amsthm}
\usepackage{tikz}
\usetikzlibrary{shapes, cd}
\usepackage{latexsym}
\usepackage{here}
\usepackage[justification=centering]{caption}

\newtheorem{theorem}{Theorem}[section]
\newtheorem{proposition}[theorem]{Proposition}
\newtheorem{corollary}[theorem]{Corollary}
\newtheorem{lemma}[theorem]{Lemma}
\theoremstyle{definition}

\newtheorem{remark}[theorem]{Remark}

\newcommand{\dia}{\Diamond}
\newcommand{\simpsimeq}{\mbox{$\diagup \! \! \! \! \: \! \: \searrow \: $}}
\newcommand{\del}{\mathrm{Del}}
\newcommand{\add}{\mathrm{Add}}

\begin{document}
\title[]{Simple Homotopy Types of Independence Complexes of Graphs involving Grid Graphs}
\author[]{Kengo Okura}
\address{Sagano High School, 15, Tokiwadannoue-cho, Ukyo-ku, Kyoto, Japan}
\email{okura.kengo.k35@kyoto-u.jp}

\begin{abstract}
We show that if a graph $G$ involves a certain square grid graph as a full subgraph, then a certain operation on it yields a simplicial suspension of the independence complex of $G$. This generalizes a result of Csorba. As a corollary, we determine the simple homotopy types of the independence complexes of some grid graphs.
\end{abstract}
\maketitle

\section{Introduction}
\label{introduction}
Let $G$ be a finite graph. Recall that an independent set of $G$ is a set of vertices such that any two elements are not joined by an edge. Then the set of independent sets of $G$ is closed under taking subsets, and so it forms a simplicial complex which is called the {\it independence complex} of $G$ and denoted by $I(G)$. Independence complexes are well studied in combinatorial algebraic topology, and in particular, there are several results on an explicit determination of the homotopy types of the independence complexes of grid graphs. In this paper, we study the simple homotopy types, instead of the homotopy types, of the independence complexes of graphs involving grid graphs.

We first refine a result of Csorba \cite[Theorem 11]{Csorba09}, an useful tool for determining the homotopy types of independence complexes, by replacing a homotopy equivalence with a simple homotopy equivalence. If simplicial complexes $K$ and $L$ have the same simple homotopy type, then we write $K \simpsimeq L$.

\begin{theorem}
\label{mainm1}
Let $G$ be a graph with an edge $uv$. If $H$ is obtained from $G$ by replacing the edge $uv$ with a length four path $u-x-y-z-v$, then
$$I(H) \simpsimeq \Sigma I(G). $$
\end{theorem}

The planer square grid graph $P_{m,n}$ is the graph of the rectangular arrangement of $(m-1)(n-1)$ squares on a plane. That is,
\begin{align*}
&V(P_{m,n}) = \{ 1, 2, \ldots ,m \} \times \{ 1, 2, \ldots ,n \} , \\
&E(P_{m,n}) = \left\{ (i_1,j_1)(i_2,j_2) \  \middle| \ 
\begin{aligned}
&i_1 = i_2 \text{ and } j_2=j_1+1 \leq n, \\
&\text{ or } i_2=i_1+1 \leq m \text{ and } j_1 = j_2 
\end{aligned} \right\}.
\end{align*}
Here we interpret Theorem \ref{mainm1} as a behavior of the simple homotopy types of $I(G)$ with respect to a replacement of a small grid graph $P_{1,2}$ in $G$ with a large grid graph $P_{1,4}$. We next generalize Theorem \ref{mainm1} from this viewpoint.

\begin{theorem}
\label{mainm2}
Let $G$ be a graph with $P_{2,2}$ as a full subgraph. If $H$ is obtained from $G$ by replacing $P_{2,2}$ with $P_{2,4}$ as in Figure \ref{mainfigure2}, then
$$I(H) \simpsimeq \Sigma I(G).$$
\begin{figure}[H]
\centering
\begin{tabular}{ccc}
{\begin{tikzpicture}
[scale=0.75]
\foreach \x in {1,4} {\foreach \y in {1,2} {\node at (\x,\y) [fill,circle,inner sep=1.0pt] {};};};
\draw (1,1)--(4,1) (1,2)--(4,2);
\draw (1,1)--(1,2) (4,1)--(4,2);
\draw (0.5,1.7)--(1,2) (0.5,2.3)--(1,2) (0.5,0.7)--(1,1) (4,2)--(4.5, 1.7) (4,2)--(4.5,2) (4,2)--(4.5,2.3) (4,1)--(4.5,1.3);
\node at (2.5,2.5) {$G$};
\end{tikzpicture}}
&{\begin{tikzpicture}[scale=0.75]\node at (0,2.3) {};\node at (0,0.7) {};\node at (0,1.5) {$\longrightarrow$};\end{tikzpicture}}&
{\begin{tikzpicture}
[scale=0.75]
\foreach \x in {1,2,3,4} {\foreach \y in {1,2} {\node at (\x,\y) [fill,circle,inner sep=1.0pt] {};};};
\draw (1,1)--(2,1)--(3,2)--(4,2);
\draw (1,2)--(2,2)--(3,1)--(4,1);
\draw (1,1)--(1,2) (2,1)--(2,2) (3,1)--(3,2) (4,1)--(4,2);
\draw (0.5,1.7)--(1,2) (0.5,2.3)--(1,2) (0.5,0.7)--(1,1) (4,2)--(4.5, 1.7) (4,2)--(4.5,2) (4,2)--(4.5,2.3) (4,1)--(4.5,1.3);
\node at (2.5,2.5) {$H$};
\end{tikzpicture}}
\end{tabular}
\caption{The replacement in Theorem \ref{mainm2}}
\label{mainfigure2}
\end{figure}
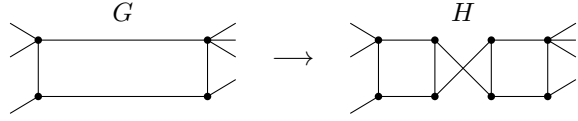
\end{theorem}

\begin{theorem}
\label{mainm3}
Let $G$ be a graph with $P_{3,2}$ as a full subgraph. If $H$ is obtained from $G$ by replacing $P_{3,2}$ with $P_{3,6}$ as in Figure \ref{mainfigure3}, then
$$I(H) \simpsimeq \Sigma^3 I(G).$$
\end{theorem}
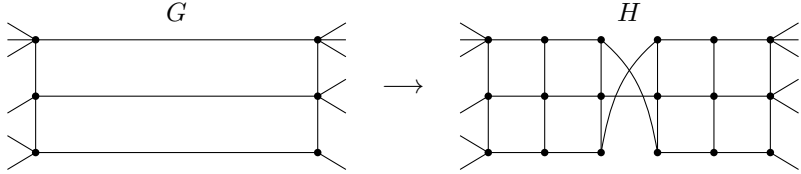
\begin{figure}[H]
\centering
\begin{tabular}{ccc}
{\begin{tikzpicture}[scale=0.75]
\foreach \x in {0,5} {\foreach \y in {1,2,3} {\node at (\x,\y) [fill,circle,inner sep=1.0pt] {};};};
\draw (0,1)--(0,2)--(0,3) (5,1)--(5,2)--(5,3);
\draw (-0.5,1.3)--(0,1) (-0.5,0.7)--(0,1) (5,3)--(5.5,3.3) (5,3)--(5.5,3) (5,3)--(5.5,2.7);
\draw (-0.5,1.7)--(0,2) (5,2)--(5.5,1.7) (5,2)--(5.5,2.3);
\draw (-0.5,3)--(0,3) (-0.5,2.7)--(0,3) (-0.5,3.3)--(0,3) (5,1)--(5.5,0.7);
\draw (0,1)--(5,1) (0,2)--(5,2) (0,3)--(5,3);
\node at (2.5,3.5) {$G$};
\end{tikzpicture}}
&{\begin{tikzpicture}[scale=0.75]\node at (0,3.3) {};\node at (0,0.7) {};\node at (0,2) {$\longrightarrow$};\end{tikzpicture}}&
{\begin{tikzpicture}[scale=0.75]
\foreach \x in {0,1,2,3,4,5} {\foreach \y in {1,2,3} {\node at (\x,\y) [fill,circle,inner sep=1.0pt] {};};};
\draw (0,1)--(0,2)--(0,3) (1,1)--(1,2)--(1,3) (2,1)--(2,2)--(2,3) (3,1)--(3,2)--(3,3) (4,1)--(4,2)--(4,3) (5,1)--(5,2)--(5,3);
\draw (0,1)--(1,1)--(2,1) (3,3)--(4,3)--(5,3);
\draw (0,2)--(1,2)--(2,2)--(3,2)--(4,2)--(5,2);
\draw (0,3)--(1,3)--(2,3) (3,1)--(4,1)--(5,1);
\draw (2,1) to [out=80, in=220] (3,3);
\draw (2,3) to [out=320, in=100] (3,1);
\draw (-0.5,1.3)--(0,1) (-0.5,0.7)--(0,1) (5,3)--(5.5,3.3) (5,3)--(5.5,3) (5,3)--(5.5,2.7);
\draw (-0.5,1.7)--(0,2) (5,2)--(5.5,1.7) (5,2)--(5.5,2.3);
\draw (-0.5,3)--(0,3) (-0.5,2.7)--(0,3) (-0.5,3.3)--(0,3) (5,1)--(5.5,0.7);
\node at (2.5,3.5) {$H$};
\end{tikzpicture}}
\end{tabular}
\caption{The replacement in Theorem \ref{mainm3} }
\label{mainfigure3}
\end{figure}

Now we apply the above results to determine the simple homotopy types of the independence complexes of the following grid graphs. 
Let $C_{m,n}$ be the graph obtained from $P_{m,n+1}$ by identifying vertices $(i,1)$ and $(i,n+1)$ for $i=1,\ldots,m$ and the corresponding edges. 
The homotopy type of $I(C_{1,n})$ is determined by Kozlov \cite[Proposition 5.2]{Koz99}, and we first refine this result to simple homotopy types. 
Let ${\dia}^n$ be the join of $n+1$ copies of $\partial\Delta^1$. Then ${\dia}^n$ is a triangulation of the $n$-sphere.
\begin{corollary}
\label{Cnsimple}
We have
\begin{align*}
I(C_{1,3k+i}) \simpsimeq \left\{
\begin{aligned}
& {\dia}^{k-1} \vee {\dia}^{k-1} & &\quad (i=0), \\
& {\dia}^{k-1} & &\quad (i=1), \\
& {\dia}^{k} & &\quad (i=2) .
\end{aligned}\right. 
\end{align*}
\end{corollary}

The homotopy types of $I(C_{2,n})$ and $I(C_{3,n})$ are also determined in previous studies. Adamaszek \cite[Theorem 1.1 b)]{Adahard} obtained $I(C_{2,n}) \simeq \Sigma^2 I(C_{2,n-4})$. In addition, the homotopy types of $I(C_{m,n})$ for $n=2,3,4,5$, which complete the determination of the homotopy type of $I(C_{2,n})$, are determined by Thapper \cite[Proposition 3.1, 3.2, 3.3, 3.4]{Thap08}. 
Iriye \cite[Theorem 1.3]{Iriye12} determined the homotopy type of $I(C_{3,n})$. We next refine these results to simple homotopy types.

\begin{corollary}
\label{CnnCnnnsimple}
We have
\begin{align*}
&I(C_{2,4k+i}) \simpsimeq \left \{
\begin{aligned}
&{\bigvee}_{3} {\dia}^{2k-1} & &(i = 0), \\
&{\dia}^{2k-1} & &(i = 1), \\
&{\dia}^{2k} & &(i = 2), \\
&{\dia}^{2k+1} & &(i = 3) ,
\end{aligned} \right. 
&I(C_{3,8k+i}) \simpsimeq \left \{
\begin{aligned}
&{\bigvee}_{5} {\dia}^{6k-1} & &(i = 0), \\
&{\dia}^{6k-1} & &(i = 1), \\
&{\dia}^{6k+1} & &(i = 2,3), \\
&{\bigvee}_{3} {\dia}^{6k+2} & &(i = 4), \\
&{\dia}^{6k+3} & &(i = 5,6), \\
&{\dia}^{6k+5} & &(i = 7) .
\end{aligned} \right.
\end{align*}
\end{corollary}

Let $M_{m,n}$ be the graph obtained from $P_{m,n+1}$ by identifying vertices $(i,1)$ and $(m-i+1,n+1)$ for $i=1,\ldots,m$ and the corresponding edges (see Figure \ref{P34C34M34figure}).
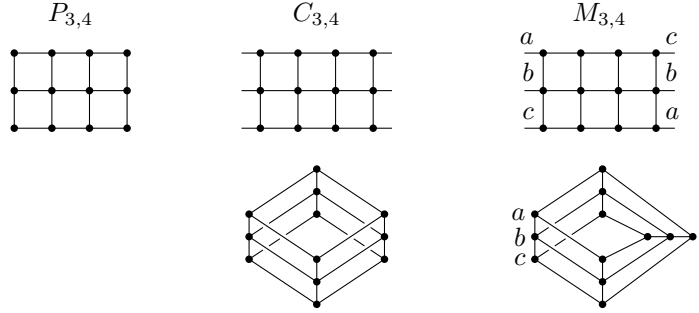
\begin{figure}[H]
\centering
\begin{tabular}{ccccc}
\begin{tikzpicture}
[scale=0.5]
\draw (0,0) grid (3,2);
\foreach \y in {0,1,2} {\foreach \x in {0,1,2,3} {\node [fill,circle,inner sep=1pt] at (\x, \y) {};};};
\node at (1.5,3) {$P_{3,4}$};
\node at (0,-0.5) {};
\end{tikzpicture}
& \hspace{0.3in}  &
\begin{tikzpicture}
[scale=0.5]
\draw (-0.5,0) grid (3.5,2);
\foreach \y in {0,1,2} {\foreach \x in {0,1,2,3} {\node [fill,circle,inner sep=1pt] at (\x, \y) {};};};
\node at (1.5,3) {$C_{3,4}$};
\node at (0,-0.5) {};
\end{tikzpicture}
& \hspace{0.3in}  &
\begin{tikzpicture}
[scale=0.5]
\draw (-0.5,0) grid (3.5,2);
\foreach \y in {0,1,2} {\foreach \x in {0,1,2,3} {\node [fill,circle,inner sep=1pt] at (\x, \y) {};};};
\node at (1.5,3) {$M_{3,4}$};
\node at (0,-0.5) {};
\node at (0,2) [above left] {$a$};
\node at (0,1) [above left] {$b$};
\node at (0,0) [above left] {$c$};
\node at (3,0) [above right] {$a$};
\node at (3,1) [above right] {$b$};
\node at (3,2) [above right] {$c$};
\end{tikzpicture}
\\
& &
\begin{tikzpicture}
[scale=0.3]
\foreach \x in {(0,2),(0,3),(0,4),(-3,0),(-3,1),(-3,2),(0,-2),(0,-1),(0,0),(3,0),(3,1),(3,2)} {\node [fill,circle,inner sep=1pt] at \x {};};
\draw (3,2)--(0,4)--(-3,2) (3,1)--(0,3)--(-3,1) (3,0)--(0,2)--(-3,0) (-3,0)--(0,-2)--(3,0);
\draw (0,2)--(0,4) (-3,0)--(-3,2) (0,0)--(0,-2) (3,0)--(3,2);
\foreach \x in {(-2.25,1.5),(-2.25,0.5),(-1.5,1),(2.25,1.5),(2.25,0.5),(1.5,1)} {\node [fill=white, circle, inner sep=1pt] at \x {};};
\draw  (-3,2)--(0,0)--(3,2)  (-3,1)--(0,-1)--(3,1);
\end{tikzpicture}
&\hspace{0.3in} &
\begin{tikzpicture}
[scale=0.3]
\foreach \x in {(0,2),(0,3),(0,4),(-3,0),(-3,1),(-3,2),(0,-2),(0,-1),(0,0),(2,1),(3,1),(4,1)} {\node [fill,circle,inner sep=1pt] at \x {};};
\draw (0,2)--(0,4) (-3,0)--(-3,2) (0,0)--(0,-2) (2,1)--(4,1);
\draw (0,3)--(-3,1) (0,-1)--(3,1)--(0,3);
\draw (0,4)--(-3,2) (0,0)--(2,1)--(0,2)--(-3,0)--(0,-2)--(4,1)--(0,4);
\node [fill=white, circle, inner sep=1pt] at (-2.25,1.5) {};
\node [fill=white, circle, inner sep=1pt] at (-2.25,0.5) {};
\node [fill=white, circle, inner sep=1pt] at (-1.5,1) {};
\draw (-3,2)--(0,0) (-3,1)--(0,-1);
\node at (-3,2) [left] {$a$};
\node at (-3,1) [left] {$b$};
\node at (-3,0) [left] {$c$};
\end{tikzpicture}
\end{tabular}
\caption{$P_{3,4}$, $C_{3,4}$ and $M_{3,4}$}
\label{P34C34M34figure}
\end{figure}

\noindent
We determine the simple homotopy types of $I(M_{2,n})$ and $I(M_{3,n})$.
\begin{corollary}
\label{MnnMnnnsimple}
We have
\begin{align*}
&I(M_{2,4k+i}) \simpsimeq \left \{
\begin{aligned}
&{\dia}^{2k-1} & &(i = 0), \\
&{\dia}^{2k} & &(i = 1), \\
&{\bigvee}_{3} {\dia}^{2k} & &(i = 2), \\
&{\dia}^{2k} & &(i = 3) ,
\end{aligned} \right. 
&I(M_{3,8k+i}) \simpsimeq \left \{
\begin{aligned}
&{\bigvee}_{3} {\dia}^{6k-1} & &(i = 0), \\
&{\dia}^{6k} & &(i = 1,2), \\
&{\dia}^{6k+2} & &(i = 3), \\
&{\bigvee}_{5} {\dia}^{6k+2} & &(i = 4), \\
&{\dia}^{6k+2} & &(i = 5), \\
&{\dia}^{6k+4} & &(i = 6,7) .
\end{aligned} \right. 
\end{align*}
\end{corollary}

Let $C^H_{m,n}$ be the graph which is obtained from $C_{m+1,2n}$ by removing all the edges $(i,j)(i-1,j)$ such that $i-j$ are even (see Figure \ref{CH23figure}).
\begin{figure}[H]
\centering
\begin{tabular}{ccc}
\begin{tikzpicture}
[scale=0.4]
\draw (-0.57,1.75)--(0,2)--(0.87,1.5)--(1.73,2)--(2.6,1.5)--(3.46,2)--(4.33,1.5)--(5.2,2)--(6.06,1.5)--(6.5,1.75);
\draw (-0.57,0.25)--(0,0)--(0.87,0.5)--(1.73,0)--(2.6,0.5)--(3.46,0)--(4.33,0.5)--(5.2,0)--(6.06,0.5)--(6.5,0.25);
\draw (-0.57,-1.25)--(0,-1)--(0.87,-1.5)--(1.73,-1)--(2.6,-1.5)--(3.46,-1)--(4.33,-1.5)--(5.2,-1)--(6.06,-1.5)--(6.5,-1.25);
\draw (0,0)--(0,-1) (0.87,1.5)--(0.87,0.5) (1.73,0)--(1.73,-1) (2.6,1.5)--(2.6,0.5) (3.46,0)--(3.46,-1) (4.33,1.5)--(4.33,0.5) (5.2,-1)--(5.2,0) (6.06,0.5)--(6.06,1.5);
\foreach \x in {(0.87,1.5),(1.73,2),(2.6,1.5),(3.46,2),(4.33,1.5),(0,0),(0.87,0.5),(1.73,0),(2.6,0.5),(3.46,0),(4.33,0.5),(0,-1),(0.87,-1.5),(1.73,-1),(2.6,-1.5),(3.46,-1),(4.33,-1.5),(0,2),(5.2,2),(6.06,1.5),(5.2,0),(6.06,0.5),(5.2,-1),(6.06,-1.5)} {\node at \x  [fill, circle, black, inner sep=1pt] {};};
\node at (0,-1.5) {$a$};
\node at (5.2,-1.5) {$a$};
\end{tikzpicture}
&\hspace{0.5in} &
\begin{tikzpicture}
[scale=0.6]
\foreach \y in {1,2,3} {\foreach \x in {1,2,3,4,5,6,7,8} {\node at (\x,\y) [fill,circle,inner sep=1pt] {};}};
\draw (0.5,1)--(1,1)--(2,1)--(3,1)--(4,1)--(5,1)--(6,1)--(7,1)--(8,1)--(8.5,1);
\draw (0.5,2)--(1,2)--(2,2)--(3,2)--(4,2)--(5,2)--(6,2)--(7,2)--(8,2)--(8.5,2);
\draw (0.5,3)--(1,3)--(2,3)--(3,3)--(4,3)--(5,3)--(6,3)--(7,3)--(8,3)--(8.5,3);
\draw (1,1)--(1,2) (2,2)--(2,3) (3,1)--(3,2) (4,2)--(4,3) (5,1)--(5,2) (6,2)--(6,3) (7,1)--(7,2) (8,2)--(8,3);
\node at (1,1) [below] {$a$};
\node at (7,1) [below] {$a$};
\end{tikzpicture}
\end{tabular}
\caption{$C^H_{2,3}$}
\label{CH23figure}
\end{figure}

\noindent
We determine the simple homotopy type of $I(C^H_{1,n})$.
\begin{corollary}
\label{CHnsimple}
We have
\begin{align*}
&I(C^H_{1,2k+i}) \simpsimeq \left \{
\begin{aligned}
&{\bigvee}_{2} {\dia}^{2k-1} &\quad &(i = 0), \\
&* &\quad &(i = 1) .
\end{aligned} \right. 
\end{align*}
\end{corollary}

\section{Preliminaries}
\label{preliminaries}
Let $G$ be a graph. For a subset $W \subset V(G)$, where $V(G)$ is the vertex set of $G$, let $G \setminus W$ be the full subgraph of $G$ whose vertex set is $V(G) \setminus W$. For vertices $u, v$ of $G$, let $G \cup uv$ denote the graph obtained from $G$ by joining $u$ with $v$ by an edge. For an edge $e$ of $G$, let $G-e$ be the graph obtained from $G$ by removing $e$. For a vertex $v$ of $G$, let $N[v]$ be the union of $\{ v \}$ and the set of vertices which are adjacent to $v$. For an edge $e=uv$ of $G$, we set $N[e]=N[u] \cup N[v]$.

We will use the following lemma.
\begin{lemma}
\label{collapselemma}
Let $G$ be a graph.
\begin{enumerate}
\item[(a)] For a vertex $v$ of $G$, suppose that there is an isolated vertex in $G \setminus N[v]$. Then 
$I(G)$ collapses onto $I(G \setminus \{v\})$.
\item[(b)] For an edge $e$ of $G$, suppose that there is an isolated vertex in $G \setminus N[e]$. Then 
$I(G-e)$ collapses onto $I(G)$.
\end{enumerate}
\end{lemma}
The following proof is based on Engstr\"{o}m's proof of \cite[Lemma 3.2]{Engs09}. 
\begin{proof}
We first prove (a). Let $u$ be the isolated vertex in $G \setminus N[v]$. We set $D =\{ \sigma \in I(G) \ |\ u \notin \sigma, v \in \sigma \}$ and $U=\{ \tau \in I(G) \ |\ u \in \tau, v \in \tau \}$. For a simplex $\sigma \in D$, $\sigma \cup \{u\}$ is again a simplex of $I(G)$ since any vertex in $\sigma$ is not adjacent to $u$. Conversely, for a simplex $\tau \in U$, $\tau \setminus \{u\}$ is a simplex of $I(G)$. Therefore, we can match $\sigma \in D$ and $\sigma \cup \{u\} \in U$, and remove all the pairs $(\sigma, \sigma \cup \{u\})$ by  elementary collapse steps. The resulted complex is $I(G \setminus \{u\})$.

Similarly, we can prove (b) by replacing $D$ and $U$ with $D' =\{ \sigma \in I(G) \ |\ u \notin \sigma, v \in \sigma, w \in \sigma \}$ and $U'=\{ \tau \in I(G) \ |\ u \in \tau, v \in \tau, w \in \tau \}$, respectively, where $v, w$ are the end points of $e$.
\end{proof}

Now we define three operations.
\begin{enumerate}
\item Suppose that $G$ has vertices $u, v$ such that $u$ is isolated in $G \setminus N[v]$. Then by Lemma \ref{collapselemma}, $I(G)$ collapses onto $I(G \setminus \{v\})$, and we denote this collapse by $\del(v,u)$. In the figures, deleted vertex $v$ is indicated by a circle.
\item Suppose that $G$ has an edge $vw$ and a vertex $u$ such that $u$ is isolated in $G \setminus N[vw]$. Then by Lemma \ref{collapselemma}, $I(G)$ expands into $I(G -e)$, and we denote this expansion by $\del(vw,u)$. In the figures, deleted edge $vw$ is indicated by a dashed line.
\item Suppose that $G$ has vertices $u, v, w$ such that $u$ is isolated in $G \setminus (N[v] \cup N[w])$. Then by Lemma \ref{collapselemma}, $I(G)$ collapses onto $I(G \cup vw)$, and we denote this collapse by $\add(vw,u)$. In the figures, added edge $e$ is indicated by a thick line.
\end{enumerate}

\section{Proofs}
\begin{proof}[Proof of Theorem \ref{mainm1}]
We obtain $I(H) \simpsimeq I(G \sqcup e)$ by the following sequence of operations (see Figure \ref{mainm1figure}):
$$\add(uv,y), \del(ux,z), \del(z,x).$$
\begin{figure}[H]
\centering
\begin{tabular}{cccc}
&
{\begin{tikzpicture}
[scale=0.6]
\foreach \x in {1,2,3,4,5} {\foreach \y in {1} {\node at (\x,\y) [fill,circle,inner sep=1.0pt] {};};};
\draw (1,1)--(5,1);
\node at (1,1) [below] {$u$};
\node at (2,1) [below] {$x$};
\node at (3,1) [below] {$y$};
\node at (4,1) [below] {$z$};
\node at (5,1) [below] {$v$};
\node at (3,2) {$H$};
\draw (0.5,0.7)--(1,1) (0.5,1.3)--(1,1) (5,1)--(5.5,0.7) (5,1)--(5.5,1) (5,1)--(5.5,1.3);
\end{tikzpicture}}
&\raisebox{3mm}{$\rightarrow$}&
{\begin{tikzpicture}
[scale=0.6]
\foreach \x in {1,2,3,4,5} {\foreach \y in {1} {\node at (\x,\y) [fill,circle,inner sep=1.0pt] {};};};
\draw (1,1)--(5,1);
\node at (1,1) [below] {$u$};
\node at (2,1) [below] {$x$};
\node at (3,1) [below] {$y$};
\node at (4,1) [below] {$z$};
\node at (5,1) [below] {$v$};
\node at (3,2) {};
\draw (0.5,0.7)--(1,1) (0.5,1.3)--(1,1) (5,1)--(5.5,0.7) (5,1)--(5.5,1) (5,1)--(5.5,1.3);
\draw [ultra thick] (1,1) to [out=20, in=160] (5,1);
\end{tikzpicture}}
\\
\raisebox{3mm}{$\rightarrow$}&
{\begin{tikzpicture}
[scale=0.6]
\foreach \x in {1,2,3,4,5} {\foreach \y in {1} {\node at (\x,\y) [fill,circle,inner sep=1.0pt] {};};};
\draw [dashed] (1,1)--(2,1);
\draw (2,1)--(5,1);
\node at (1,1) [below] {$u$};
\node at (2,1) [below] {$x$};
\node at (3,1) [below] {$y$};
\node at (4,1) [below] {$z$};
\node at (5,1) [below] {$v$};
\node at (3,2) {};
\draw (0.5,0.7)--(1,1) (0.5,1.3)--(1,1) (5,1)--(5.5,0.7) (5,1)--(5.5,1) (5,1)--(5.5,1.3);
\draw (1,1) to [out=20, in=160] (5,1);
\end{tikzpicture}}
&\raisebox{3mm}{$\rightarrow$}&
{\begin{tikzpicture}
[scale=0.6]
\foreach \x in {1,2,3,5} {\foreach \y in {1} {\node at (\x,\y) [fill,circle,inner sep=1.0pt] {};};};
\draw [dashed] (3,1)--(5,1);
\draw (2,1)--(3,1);
\node at (1,1) [below] {$u$};
\node at (2,1) [below] {$x$};
\node at (3,1) [below] {$y$};
\node at (4,1) [below] {$z$};
\node at (5,1) [below] {$v$};
\node at (3,2) {$G \sqcup e$};
\draw (0.5,0.7)--(1,1) (0.5,1.3)--(1,1) (5,1)--(5.5,0.7) (5,1)--(5.5,1) (5,1)--(5.5,1.3);
\draw (1,1) to [out=20, in=160] (5,1);
\node at (4,1) [circle, draw=black, fill=white, inner sep=1pt] {};
\end{tikzpicture}}
\end{tabular}
\caption{A sequence of operations which induces $I(H) \simpsimeq I(G \sqcup e)$ in the proof of Theorem \ref{mainm1}}
\label{mainm1figure}
\end{figure}
Then, we get the desired conclusion since we have $I(G \sqcup e) \simpsimeq \Sigma I(G)$.
\end{proof}

\begin{proof}[Proof of Corollary \ref{Cnsimple}]
By Theorem \ref{mainm1}, we have
$$I(C_{1,n+3}) \simpsimeq \Sigma I(C_{1,n}) .$$
The base cases are
\begin{align*}
&I(C_{1,1}) = {\dia}^{-1}, & &I(C_{1,2}) = {\dia}^0, & &I(C_{1,3}) = {\dia}^0 \vee {\dia}^0
\end{align*}
(see Figure \ref{mainm1basecases}).
\begin{figure}[H]
\centering
\begin{tabular}{ccccc}
\begin{tikzpicture}
[scale=0.6]
\draw (1,1) to [out=0, in=0] (1,1.5); 
\draw (1,1.5) to [out=180, in=180] (1,1);
\node at (1,1) [fill, circle, black, inner sep=1pt] {};
\node at (1,2) {$C_{1,1}$};
\node at (0,0.5) {};
\end{tikzpicture}& \hspace{0.2in} &
\begin{tikzpicture}
[scale=0.6]
\foreach \x in {(1,1),(2,1)} {\node at \x [fill, circle, black, inner sep=1pt] {};};
\draw (1,1)--(2,1);
\draw (1,1) to [out=-45, in=-135] (2,1);
\node at (1.5,2) {$C_{1,2}$};
\node at (0,0.5) {};
\end{tikzpicture}& \hspace{0.2in} & 
\begin{tikzpicture}
[scale=0.6]
\draw (1,1)--(3,1);
\foreach \x in {(1,1),(2,1),(3,1)} {\node at \x [fill, circle, black, inner sep=1pt] {};};
\draw (1,1) to [out=-45, in=-135] (3,1);
\node at (2,2) {$C_{1,3}$};
\node at (0,0.5) {};
\end{tikzpicture}
\\
\begin{tikzpicture}
[scale=0.6]
\node at (1,1) {$\emptyset$};
\node at (1,2) {$I(C_{1,1})$};
\node at (0,0.5) {};
\end{tikzpicture}& \hspace{0.2in} &
\begin{tikzpicture}
[scale=0.6]
\foreach \x in {(1,1),(2,1)} {\node at \x [fill, circle, black, inner sep=1pt] {};};
\node at (1.5,2) {$I(C_{1,2})$};
\node at (0,0.5) {};
\end{tikzpicture}& \hspace{0.2in} & 
\begin{tikzpicture}
[scale=0.6]
\foreach \x in {(1,1),(2,1),(3,1)} {\node at \x [fill, circle, black, inner sep=1pt] {};};
\node at (2,2) {$I(C_{1,3})$};
\node at (0,0.5) {};
\end{tikzpicture}
\end{tabular}
\caption{$C_{1,1}$, $C_{1,2}$, $C_{1,3}$ and their independence complexes}
\label{mainm1basecases}
\end{figure}
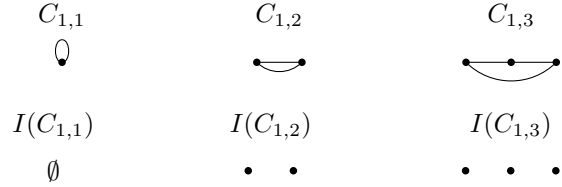
\end{proof}

In the following, we denote the vertices $(1,i), (2,i), (3,i), (4,i)$ of $P_{m,n}$ by $i, \overline{i}, \widehat{i}$, $\widetilde{i}$, respectively.
\begin{proof}[Proof of Theorem \ref{mainm2}]
The desired conclusion is $I(H) \simpsimeq I(G \sqcup e).$ This is obtained by
$$\add(1\overline{4},3), \add(\overline{1}4, \overline{3}), \del(12,\overline{3}), \del(\overline{1}\overline{2},3), \del(3,\overline{2}), \del(\overline{3}, 2)$$
(see Figure \ref{mainm2figure}).
\begin{figure}[H]
\centering
\begin{tabular}{cccccccc}
&
{\begin{tikzpicture}
[scale=0.6]
\foreach \x in {1,2,3,4} {\foreach \y in {1,2} {\node at (\x,\y) [fill,circle,inner sep=1pt] {};};};
\draw (1,1)--(2,1)--(3,1)--(4,1);
\draw (1,2)--(2,2)--(3,2)--(4,2);
\draw (1,1)--(1,2) (2,1)--(2,2) (3,1)--(3,2) (4,1)--(4,2);
\draw (0.5,1.7)--(1,2) (0.5,2.3)--(1,2) (0.5,0.7)--(1,1) (4,1)--(4.5,0.7) (4,1)--(4.5,1) (4,1)--(4.5,1.3) (4,2)--(4.5,1.7);
\end{tikzpicture}}
&\raisebox{3mm}{$\rightarrow$}&
{\begin{tikzpicture}
[scale=0.6]
\foreach \x in {1,2,3,4} {\foreach \y in {1,2} {\node at (\x,\y) [fill,circle,inner sep=1pt] {};};};
\draw (1,1)--(2,1)--(3,1)--(4,1);
\draw (1,2)--(2,2)--(3,2)--(4,2);
\draw (1,1)--(1,2) (2,1)--(2,2) (3,1)--(3,2) (4,1)--(4,2);
\draw (0.5,1.7)--(1,2) (0.5,2.3)--(1,2) (0.5,0.7)--(1,1) (4,1)--(4.5,0.7) (4,1)--(4.5,1) (4,1)--(4.5,1.3) (4,2)--(4.5,1.7);
\draw [ultra thick] (1,1)--(4,2);
\end{tikzpicture}}
&\raisebox{3mm}{$\rightarrow$}&
{\begin{tikzpicture}
[scale=0.6]
\foreach \x in {1,2,3,4} {\foreach \y in {1,2} {\node at (\x,\y) [fill,circle,inner sep=1pt] {};};};
\draw (1,1)--(2,1)--(3,1)--(4,1);
\draw (1,2)--(2,2)--(3,2)--(4,2);
\draw (1,1)--(1,2) (2,1)--(2,2) (3,1)--(3,2) (4,1)--(4,2);
\draw (0.5,1.7)--(1,2) (0.5,2.3)--(1,2) (0.5,0.7)--(1,1) (4,1)--(4.5,0.7) (4,1)--(4.5,1) (4,1)--(4.5,1.3) (4,2)--(4.5,1.7);
\draw (1,1)--(4,2);
\draw [ultra thick] (1,2)--(4,1);
\end{tikzpicture}}
\\
\raisebox{3mm}{$\rightarrow$}&
{\begin{tikzpicture}
[scale=0.6]
\foreach \x in {1,2,3,4} {\foreach \y in {1,2} {\node at (\x,\y) [fill,circle,inner sep=1pt] {};};};
\draw (2,1)--(3,1)--(4,1);
\draw (1,2)--(2,2)--(3,2)--(4,2);
\draw (1,1)--(1,2) (2,1)--(2,2) (3,1)--(3,2) (4,1)--(4,2);
\draw (0.5,1.7)--(1,2) (0.5,2.3)--(1,2) (0.5,0.7)--(1,1) (4,1)--(4.5,0.7) (4,1)--(4.5,1) (4,1)--(4.5,1.3) (4,2)--(4.5,1.7);
\draw (1,1)--(4,2) (1,2)--(4,1);
\draw [dashed] (1,1)--(2,1);
\end{tikzpicture}}
&\raisebox{3mm}{$\rightarrow$}&
{\begin{tikzpicture}
[scale=0.6]
\foreach \x in {1,2,3,4} {\foreach \y in {1,2} {\node at (\x,\y) [fill,circle,inner sep=1pt] {};};};
\draw (2,1)--(3,1)--(4,1);
\draw (2,2)--(3,2)--(4,2);
\draw (1,1)--(1,2) (2,1)--(2,2) (3,1)--(3,2) (4,1)--(4,2);
\draw (0.5,1.7)--(1,2) (0.5,2.3)--(1,2) (0.5,0.7)--(1,1) (4,1)--(4.5,0.7) (4,1)--(4.5,1) (4,1)--(4.5,1.3) (4,2)--(4.5,1.7);
\draw (1,1)--(4,2) (1,2)--(4,1);
\draw [dashed] (1,2)--(2,2);
\end{tikzpicture}}
&\raisebox{3mm}{$\rightarrow$}&
{\begin{tikzpicture}
[scale=0.6]
\foreach \x in {1,2,4} {\foreach \y in {1,2} {\node at (\x,\y) [fill,circle,inner sep=1pt] {};};};
\node at (3,2) [fill,circle,inner sep=1pt] {};
\draw (2,2)--(3,2)--(4,2);
\draw (1,1)--(1,2) (2,1)--(2,2) (4,1)--(4,2);
\draw (0.5,1.7)--(1,2) (0.5,2.3)--(1,2) (0.5,0.7)--(1,1) (4,1)--(4.5,0.7) (4,1)--(4.5,1) (4,1)--(4.5,1.3) (4,2)--(4.5,1.7);
\draw (1,1)--(4,2) (1,2)--(4,1);
\draw [dashed] (2,1)--(4,1) (3,1)--(3,2);
\node at (3,1) [circle, draw=black, fill=white, inner sep=1pt] {};
\end{tikzpicture}}
\\
\raisebox{3mm}{$\rightarrow$}&
{\begin{tikzpicture}
[scale=0.6]
\foreach \x in {1,2,4} {\foreach \y in {1,2} {\node at (\x,\y) [fill,circle,inner sep=1pt] {};};};
\draw (1,1)--(1,2) (2,1)--(2,2) (4,1)--(4,2);
\draw (0.5,1.7)--(1,2) (0.5,2.3)--(1,2) (0.5,0.7)--(1,1) (4,1)--(4.5,0.7) (4,1)--(4.5,1) (4,1)--(4.5,1.3) (4,2)--(4.5,1.7);
\draw (1,1)--(4,2) (1,2)--(4,1);
\draw [dashed] (2,2)--(4,2);
\node at (3,2) [circle, draw=black, fill=white, inner sep=1pt] {};
\end{tikzpicture}}
\end{tabular}
\caption{A sequence of operations which induces $I(H) \simpsimeq I(G \sqcup e)$ in the proof of Theorem \ref{mainm2}}
\label{mainm2figure}
\end{figure}
\end{proof}

\begin{remark}
\label{remmainm2}
The above proof of Theorem \ref{mainm2} indicates that we can apply Theorem \ref{mainm2} even if there is no edge $1\overline{1}$.
\end{remark}

\begin{proof}[Proof of Theorem \ref{mainm3}]
It is sufficient to obtain $I(H) \simpsimeq I(G \sqcup C_{1,8})$ since by Corollary \ref{Cnsimple}, we have
$$I(G \sqcup C_{1,8}) \simpsimeq I(G) * {\dia}^2 \simpsimeq {\Sigma}^3 I(G) .$$
There are four steps to transform $H$ into $G \sqcup C_{1,8}$ without changing the simple homotopy type of the independence complex.
\begin{description}
\item[Step 1] Add three edges, $\overline{1}\overline{6}$, $1\widehat{6}$ and $\widehat{1}6$. This process is performed as follows:
\begin{align*}
&\add(\overline{1}\widehat{4},\widehat{2}), \add(\overline{1}\overline{6},\widehat{5}), \del(\overline{1}\widehat{4},\widehat{2}),
\add(1\widehat{3},\overline{2}), \add(1\overline{4},3),\\
&\add(1\widehat{6},\widehat{4}), \del(1\overline{4},3), \del(1\widehat{3},\overline{2}), \add(\widehat{1}3,\overline{2}), \add(\widehat{1}\overline{4},\widehat{3}),\\
&\add(\widehat{1}6,4), \del(\widehat{1}\overline{4},\widehat{3}), \del(\widehat{1}3,\overline{2})
\end{align*}
(see Figure \ref{mainm3step1}).
\begin{figure}[H]
\centering
\begin{tabular}{cccccc}
&
{\begin{tikzpicture}[scale=0.5]
\foreach \x in {0,1,2,3,4,5} {\foreach \y in {1,2,3} {\node at (\x,\y) [fill,circle,inner sep=1pt] {};};};
\draw (0,1)--(0,2)--(0,3) (1,1)--(1,2)--(1,3) (2,1)--(2,2)--(2,3) (3,1)--(3,2)--(3,3) (4,1)--(4,2)--(4,3) (5,1)--(5,2)--(5,3);
\draw (0,1)--(1,1)--(2,1)--(3,1)--(4,1)--(5,1);
\draw (0,2)--(1,2)--(2,2)--(3,2)--(4,2)--(5,2);
\draw (0,3)--(1,3)--(2,3)--(3,3)--(4,3)--(5,3);
\draw (-0.5,1.3)--(0,1) (-0.5,0.7)--(0,1) (5,1)--(5.5,1.3) (5,1)--(5.5,1) (5,1)--(5.5,0.7);
\draw (-0.5,1.7)--(0,2) (5,2)--(5.5,1.7) (5,2)--(5.5,2.3);
\draw (-0.5,3)--(0,3) (-0.5,2.7)--(0,3) (-0.5,3.3)--(0,3) (5,3)--(5.5,3.3);
\end{tikzpicture}}
&\raisebox{7mm}{$\rightarrow$}&
{\begin{tikzpicture}[scale=0.5]
\foreach \x in {0,1,2,3,4,5} {\foreach \y in {1,2,3} {\node at (\x,\y) [fill,circle,inner sep=1pt] {};};};
\draw (0,1)--(0,2)--(0,3) (1,1)--(1,2)--(1,3) (2,1)--(2,2)--(2,3) (3,1)--(3,2)--(3,3) (4,1)--(4,2)--(4,3) (5,1)--(5,2)--(5,3);
\draw (0,1)--(1,1)--(2,1)--(3,1)--(4,1)--(5,1);
\draw (0,2)--(1,2)--(2,2)--(3,2)--(4,2)--(5,2);
\draw (0,3)--(1,3)--(2,3)--(3,3)--(4,3)--(5,3);
\draw (-0.5,1.3)--(0,1) (-0.5,0.7)--(0,1) (5,1)--(5.5,1.3) (5,1)--(5.5,1) (5,1)--(5.5,0.7);
\draw (-0.5,1.7)--(0,2) (5,2)--(5.5,1.7) (5,2)--(5.5,2.3);
\draw (-0.5,3)--(0,3) (-0.5,2.7)--(0,3) (-0.5,3.3)--(0,3) (5,3)--(5.5,3.3);
\draw [ultra thick] (0,2) to [out=40, in=190] (3,3);
\end{tikzpicture}}
&\raisebox{7mm}{$\rightarrow$}&
{\begin{tikzpicture}[scale=0.5]
\foreach \x in {0,1,2,3,4,5} {\foreach \y in {1,2,3} {\node at (\x,\y) [fill,circle,inner sep=1pt] {};};};
\draw (0,1)--(0,2)--(0,3) (1,1)--(1,2)--(1,3) (2,1)--(2,2)--(2,3) (3,1)--(3,2)--(3,3) (4,1)--(4,2)--(4,3) (5,1)--(5,2)--(5,3);
\draw (0,1)--(1,1)--(2,1)--(3,1)--(4,1)--(5,1);
\draw (0,2)--(1,2)--(2,2)--(3,2)--(4,2)--(5,2);
\draw (0,3)--(1,3)--(2,3)--(3,3)--(4,3)--(5,3);
\draw (-0.5,1.3)--(0,1) (-0.5,0.7)--(0,1) (5,1)--(5.5,1.3) (5,1)--(5.5,1) (5,1)--(5.5,0.7);
\draw (-0.5,1.7)--(0,2) (5,2)--(5.5,1.7) (5,2)--(5.5,2.3);
\draw (-0.5,3)--(0,3) (-0.5,2.7)--(0,3) (-0.5,3.3)--(0,3) (5,3)--(5.5,3.3);
\draw (0,2) to [out=40, in=190] (3,3);
\draw [ultra thick] (0,2) to [out=20, in=160] (5,2);
\end{tikzpicture}}
\\
\raisebox{7mm}{$\rightarrow$}&
{\begin{tikzpicture}[scale=0.5]
\foreach \x in {0,1,2,3,4,5} {\foreach \y in {1,2,3} {\node at (\x,\y) [fill,circle,inner sep=1pt] {};};};
\draw (0,1)--(0,2)--(0,3) (1,1)--(1,2)--(1,3) (2,1)--(2,2)--(2,3) (3,1)--(3,2)--(3,3) (4,1)--(4,2)--(4,3) (5,1)--(5,2)--(5,3);
\draw (0,1)--(1,1)--(2,1)--(3,1)--(4,1)--(5,1);
\draw (0,2)--(1,2)--(2,2)--(3,2)--(4,2)--(5,2);
\draw (0,3)--(1,3)--(2,3)--(3,3)--(4,3)--(5,3);
\draw (-0.5,1.3)--(0,1) (-0.5,0.7)--(0,1) (5,1)--(5.5,1.3) (5,1)--(5.5,1) (5,1)--(5.5,0.7);
\draw (-0.5,1.7)--(0,2) (5,2)--(5.5,1.7) (5,2)--(5.5,2.3);
\draw (-0.5,3)--(0,3) (-0.5,2.7)--(0,3) (-0.5,3.3)--(0,3) (5,3)--(5.5,3.3);
\draw (0,2) to [out=20, in=160] (5,2);
\draw [dashed] (0,2) to [out=40, in=190] (3,3);
\end{tikzpicture}}
&\raisebox{7mm}{$\rightarrow$}&
{\begin{tikzpicture}[scale=0.5]
\foreach \x in {0,1,2,3,4,5} {\foreach \y in {1,2,3} {\node at (\x,\y) [fill,circle,inner sep=1pt] {};};};
\draw (0,1)--(0,2)--(0,3) (1,1)--(1,2)--(1,3) (2,1)--(2,2)--(2,3) (3,1)--(3,2)--(3,3) (4,1)--(4,2)--(4,3) (5,1)--(5,2)--(5,3);
\draw (0,1)--(1,1)--(2,1)--(3,1)--(4,1)--(5,1);
\draw (0,2)--(1,2)--(2,2)--(3,2)--(4,2)--(5,2);
\draw (0,3)--(1,3)--(2,3)--(3,3)--(4,3)--(5,3);
\draw (-0.5,1.3)--(0,1) (-0.5,0.7)--(0,1) (5,1)--(5.5,1.3) (5,1)--(5.5,1) (5,1)--(5.5,0.7);
\draw (-0.5,1.7)--(0,2) (5,2)--(5.5,1.7) (5,2)--(5.5,2.3);
\draw (-0.5,3)--(0,3) (-0.5,2.7)--(0,3) (-0.5,3.3)--(0,3) (5,3)--(5.5,3.3);
\draw (0,2) to [out=20, in=160] (5,2);
\draw [ultra thick] (0,1) to [out=70, in=200] (2,3);
\end{tikzpicture}}
&\raisebox{7mm}{$\rightarrow$}&
{\begin{tikzpicture}[scale=0.5]
\foreach \x in {0,1,2,3,4,5} {\foreach \y in {1,2,3} {\node at (\x,\y) [fill,circle,inner sep=1pt] {};};};
\draw (0,1)--(0,2)--(0,3) (1,1)--(1,2)--(1,3) (2,1)--(2,2)--(2,3) (3,1)--(3,2)--(3,3) (4,1)--(4,2)--(4,3) (5,1)--(5,2)--(5,3);
\draw (0,1)--(1,1)--(2,1)--(3,1)--(4,1)--(5,1);
\draw (0,2)--(1,2)--(2,2)--(3,2)--(4,2)--(5,2);
\draw (0,3)--(1,3)--(2,3)--(3,3)--(4,3)--(5,3);
\draw (-0.5,1.3)--(0,1) (-0.5,0.7)--(0,1) (5,1)--(5.5,1.3) (5,1)--(5.5,1) (5,1)--(5.5,0.7);
\draw (-0.5,1.7)--(0,2) (5,2)--(5.5,1.7) (5,2)--(5.5,2.3);
\draw (-0.5,3)--(0,3) (-0.5,2.7)--(0,3) (-0.5,3.3)--(0,3) (5,3)--(5.5,3.3);
\draw (0,2) to [out=20, in=160] (5,2);
\draw (0,1) to [out=70, in=200] (2,3);
\draw [ultra thick] (0,1) to [out=10, in=220] (3,2);
\end{tikzpicture}}
\\
\raisebox{7mm}{$\rightarrow$}&
{\begin{tikzpicture}[scale=0.5]
\foreach \x in {0,1,2,3,4,5} {\foreach \y in {1,2,3} {\node at (\x,\y) [fill,circle,inner sep=1pt] {};};};
\draw (0,1)--(0,2)--(0,3) (1,1)--(1,2)--(1,3) (2,1)--(2,2)--(2,3) (3,1)--(3,2)--(3,3) (4,1)--(4,2)--(4,3) (5,1)--(5,2)--(5,3);
\draw (0,1)--(1,1)--(2,1)--(3,1)--(4,1)--(5,1);
\draw (0,2)--(1,2)--(2,2)--(3,2)--(4,2)--(5,2);
\draw (0,3)--(1,3)--(2,3)--(3,3)--(4,3)--(5,3);
\draw (-0.5,1.3)--(0,1) (-0.5,0.7)--(0,1) (5,1)--(5.5,1.3) (5,1)--(5.5,1) (5,1)--(5.5,0.7);
\draw (-0.5,1.7)--(0,2) (5,2)--(5.5,1.7) (5,2)--(5.5,2.3);
\draw (-0.5,3)--(0,3) (-0.5,2.7)--(0,3) (-0.5,3.3)--(0,3) (5,3)--(5.5,3.3);
\draw (0,2) to [out=20, in=160] (5,2);
\draw [ultra thick] (0,1)--(5,3);
\draw (0,1) to [out=70, in=200] (2,3);
\draw (0,1) to [out=10, in=220] (3,2);
\end{tikzpicture}}
&\raisebox{7mm}{$\rightarrow$}&
{\begin{tikzpicture}[scale=0.5]
\foreach \x in {0,1,2,3,4,5} {\foreach \y in {1,2,3} {\node at (\x,\y) [fill,circle,inner sep=1pt] {};};};
\draw (0,1)--(0,2)--(0,3) (1,1)--(1,2)--(1,3) (2,1)--(2,2)--(2,3) (3,1)--(3,2)--(3,3) (4,1)--(4,2)--(4,3) (5,1)--(5,2)--(5,3);
\draw (0,1)--(1,1)--(2,1)--(3,1)--(4,1)--(5,1);
\draw (0,2)--(1,2)--(2,2)--(3,2)--(4,2)--(5,2);
\draw (0,3)--(1,3)--(2,3)--(3,3)--(4,3)--(5,3);
\draw (-0.5,1.3)--(0,1) (-0.5,0.7)--(0,1) (5,1)--(5.5,1.3) (5,1)--(5.5,1) (5,1)--(5.5,0.7);
\draw (-0.5,1.7)--(0,2) (5,2)--(5.5,1.7) (5,2)--(5.5,2.3);
\draw (-0.5,3)--(0,3) (-0.5,2.7)--(0,3) (-0.5,3.3)--(0,3) (5,3)--(5.5,3.3);
\draw (0,2) to [out=20, in=160] (5,2);
\draw (0,1)--(5,3);
\draw (0,1) to [out=70, in=200] (2,3);
\draw [dashed] (0,1) to [out=10, in=220] (3,2);
\end{tikzpicture}}
&\raisebox{7mm}{$\rightarrow$}&
{\begin{tikzpicture}[scale=0.5]
\foreach \x in {0,1,2,3,4,5} {\foreach \y in {1,2,3} {\node at (\x,\y) [fill,circle,inner sep=1pt] {};};};
\draw (0,1)--(0,2)--(0,3) (1,1)--(1,2)--(1,3) (2,1)--(2,2)--(2,3) (3,1)--(3,2)--(3,3) (4,1)--(4,2)--(4,3) (5,1)--(5,2)--(5,3);
\draw (0,1)--(1,1)--(2,1)--(3,1)--(4,1)--(5,1);
\draw (0,2)--(1,2)--(2,2)--(3,2)--(4,2)--(5,2);
\draw (0,3)--(1,3)--(2,3)--(3,3)--(4,3)--(5,3);
\draw (-0.5,1.3)--(0,1) (-0.5,0.7)--(0,1) (5,1)--(5.5,1.3) (5,1)--(5.5,1) (5,1)--(5.5,0.7);
\draw (-0.5,1.7)--(0,2) (5,2)--(5.5,1.7) (5,2)--(5.5,2.3);
\draw (-0.5,3)--(0,3) (-0.5,2.7)--(0,3) (-0.5,3.3)--(0,3) (5,3)--(5.5,3.3);
\draw (0,2) to [out=20, in=160] (5,2);
\draw (0,1)--(5,3);
\draw [dashed] (0,1) to [out=70, in=200] (2,3);
\end{tikzpicture}}
\\
\raisebox{7mm}{$\rightarrow$}&
{\begin{tikzpicture}[scale=0.5]
\foreach \x in {0,1,2,3,4,5} {\foreach \y in {1,2,3} {\node at (\x,\y) [fill,circle,inner sep=1pt] {};};};
\draw (0,1)--(0,2)--(0,3) (1,1)--(1,2)--(1,3) (2,1)--(2,2)--(2,3) (3,1)--(3,2)--(3,3) (4,1)--(4,2)--(4,3) (5,1)--(5,2)--(5,3);
\draw (0,1)--(1,1)--(2,1)--(3,1)--(4,1)--(5,1);
\draw (0,2)--(1,2)--(2,2)--(3,2)--(4,2)--(5,2);
\draw (0,3)--(1,3)--(2,3)--(3,3)--(4,3)--(5,3);
\draw (-0.5,1.3)--(0,1) (-0.5,0.7)--(0,1) (5,1)--(5.5,1.3) (5,1)--(5.5,1) (5,1)--(5.5,0.7);
\draw (-0.5,1.7)--(0,2) (5,2)--(5.5,1.7) (5,2)--(5.5,2.3);
\draw (-0.5,3)--(0,3) (-0.5,2.7)--(0,3) (-0.5,3.3)--(0,3) (5,3)--(5.5,3.3);
\draw (0,2) to [out=20, in=160] (5,2);
\draw (0,1)--(5,3);
\draw [ultra thick] (0,3) to [out=290, in=160] (2,1);
\end{tikzpicture}}
&\raisebox{7mm}{$\rightarrow$}&
{\begin{tikzpicture}[scale=0.5]
\foreach \x in {0,1,2,3,4,5} {\foreach \y in {1,2,3} {\node at (\x,\y) [fill,circle,inner sep=1pt] {};};};
\draw (0,1)--(0,2)--(0,3) (1,1)--(1,2)--(1,3) (2,1)--(2,2)--(2,3) (3,1)--(3,2)--(3,3) (4,1)--(4,2)--(4,3) (5,1)--(5,2)--(5,3);
\draw (0,1)--(1,1)--(2,1)--(3,1)--(4,1)--(5,1);
\draw (0,2)--(1,2)--(2,2)--(3,2)--(4,2)--(5,2);
\draw (0,3)--(1,3)--(2,3)--(3,3)--(4,3)--(5,3);
\draw (-0.5,1.3)--(0,1) (-0.5,0.7)--(0,1) (5,1)--(5.5,1.3) (5,1)--(5.5,1) (5,1)--(5.5,0.7);
\draw (-0.5,1.7)--(0,2) (5,2)--(5.5,1.7) (5,2)--(5.5,2.3);
\draw (-0.5,3)--(0,3) (-0.5,2.7)--(0,3) (-0.5,3.3)--(0,3) (5,3)--(5.5,3.3);
\draw (0,2) to [out=20, in=160] (5,2);
\draw (0,1)--(5,3);
\draw (0,3) to [out=290, in=160] (2,1);
\draw [ultra thick] (0,3) to [out=350, in=130] (3,2);
\end{tikzpicture}}
&\raisebox{7mm}{$\rightarrow$}&
{\begin{tikzpicture}[scale=0.5]
\foreach \x in {0,1,2,3,4,5} {\foreach \y in {1,2,3} {\node at (\x,\y) [fill,circle,inner sep=1pt] {};};};
\draw (0,1)--(0,2)--(0,3) (1,1)--(1,2)--(1,3) (2,1)--(2,2)--(2,3) (3,1)--(3,2)--(3,3) (4,1)--(4,2)--(4,3) (5,1)--(5,2)--(5,3);
\draw (0,1)--(1,1)--(2,1)--(3,1)--(4,1)--(5,1);
\draw (0,2)--(1,2)--(2,2)--(3,2)--(4,2)--(5,2);
\draw (0,3)--(1,3)--(2,3)--(3,3)--(4,3)--(5,3);
\draw (-0.5,1.3)--(0,1) (-0.5,0.7)--(0,1) (5,1)--(5.5,1.3) (5,1)--(5.5,1) (5,1)--(5.5,0.7);
\draw (-0.5,1.7)--(0,2) (5,2)--(5.5,1.7) (5,2)--(5.5,2.3);
\draw (-0.5,3)--(0,3) (-0.5,2.7)--(0,3) (-0.5,3.3)--(0,3) (5,3)--(5.5,3.3);
\draw (0,2) to [out=20, in=160] (5,2);
\draw (0,1)--(5,3);
\draw [ultra thick] (0,3)--(5,1);
\draw (0,3) to [out=290, in=160] (2,1);
\draw (0,3) to [out=350, in=130] (3,2);
\end{tikzpicture}}
\\
\raisebox{7mm}{$\rightarrow$}&
{\begin{tikzpicture}[scale=0.5]
\foreach \x in {0,1,2,3,4,5} {\foreach \y in {1,2,3} {\node at (\x,\y) [fill,circle,inner sep=1pt] {};};};
\draw (0,1)--(0,2)--(0,3) (1,1)--(1,2)--(1,3) (2,1)--(2,2)--(2,3) (3,1)--(3,2)--(3,3) (4,1)--(4,2)--(4,3) (5,1)--(5,2)--(5,3);
\draw (0,1)--(1,1)--(2,1)--(3,1)--(4,1)--(5,1);
\draw (0,2)--(1,2)--(2,2)--(3,2)--(4,2)--(5,2);
\draw (0,3)--(1,3)--(2,3)--(3,3)--(4,3)--(5,3);
\draw (-0.5,1.3)--(0,1) (-0.5,0.7)--(0,1) (5,1)--(5.5,1.3) (5,1)--(5.5,1) (5,1)--(5.5,0.7);
\draw (-0.5,1.7)--(0,2) (5,2)--(5.5,1.7) (5,2)--(5.5,2.3);
\draw (-0.5,3)--(0,3) (-0.5,2.7)--(0,3) (-0.5,3.3)--(0,3) (5,3)--(5.5,3.3);
\draw (0,2) to [out=20, in=160] (5,2);
\draw (0,1)--(5,3) (0,3)--(5,1);
\draw (0,3) to [out=290, in=160] (2,1);
\draw [dashed] (0,3) to [out=350, in=130] (3,2);
\end{tikzpicture}}
&\raisebox{7mm}{$\rightarrow$}&
{\begin{tikzpicture}[scale=0.5]
\foreach \x in {0,1,2,3,4,5} {\foreach \y in {1,2,3} {\node at (\x,\y) [fill,circle,inner sep=1pt] {};};};
\draw (0,1)--(0,2)--(0,3) (1,1)--(1,2)--(1,3) (2,1)--(2,2)--(2,3) (3,1)--(3,2)--(3,3) (4,1)--(4,2)--(4,3) (5,1)--(5,2)--(5,3);
\draw (0,1)--(1,1)--(2,1)--(3,1)--(4,1)--(5,1);
\draw (0,2)--(1,2)--(2,2)--(3,2)--(4,2)--(5,2);
\draw (0,3)--(1,3)--(2,3)--(3,3)--(4,3)--(5,3);
\draw (-0.5,1.3)--(0,1) (-0.5,0.7)--(0,1) (5,1)--(5.5,1.3) (5,1)--(5.5,1) (5,1)--(5.5,0.7);
\draw (-0.5,1.7)--(0,2) (5,2)--(5.5,1.7) (5,2)--(5.5,2.3);
\draw (-0.5,3)--(0,3) (-0.5,2.7)--(0,3) (-0.5,3.3)--(0,3) (5,3)--(5.5,3.3);
\draw (0,2) to [out=20, in=160] (5,2);
\draw (0,1)--(5,3) (0,3)--(5,1);
\draw [dashed] (0,3) to [out=290, in=160] (2,1);
\end{tikzpicture}}
\end{tabular}
\caption{A sequence of operations which induces $I(H) \simpsimeq I(G \sqcup C_{1,8})$ in the proof of Theorem \ref{mainm3} (Step 1)}
\label{mainm3step1}
\end{figure}

\item[Step 2] Delete three edges, $12$, $\widehat{1}\widehat{2}$ and $\overline{1}\overline{2}$. This process is performed as follows:
\begin{align*}
&\add(2\widehat{4},\overline{3}), \add(2\overline{5},4), \del(12,\widehat{5}), \del(2\widehat{4},\overline{3}), \add(\widehat{2}4,\overline{3}),\\
&\add(\widehat{2}\overline{5},\widehat{4}), \del(\widehat{1}\widehat{2},5), \del(\widehat{2}4,\overline{3}), \add(\overline{2}\widehat{5},\widehat{3}), \add(\overline{2}\overline{4},3),\\
&\add(\overline{2}5,3), \del(\overline{1}\overline{2},\overline{5})
\end{align*}
(see Figure \ref{mainm3step2}).
\begin{figure}[H]
\centering
\begin{tabular}{cccccc}
\raisebox{7mm}{$\rightarrow$}&
{\begin{tikzpicture}[scale=0.5]
\foreach \x in {0,1,2,3,4,5} {\foreach \y in {1,2,3} {\node at (\x,\y) [fill,circle,inner sep=1pt] {};};};
\draw (0,1)--(0,2)--(0,3) (1,1)--(1,2)--(1,3) (2,1)--(2,2)--(2,3) (3,1)--(3,2)--(3,3) (4,1)--(4,2)--(4,3) (5,1)--(5,2)--(5,3);
\draw (0,1)--(1,1)--(2,1)--(3,1)--(4,1)--(5,1);
\draw (0,2)--(1,2)--(2,2)--(3,2)--(4,2)--(5,2);
\draw (0,3)--(1,3)--(2,3)--(3,3)--(4,3)--(5,3);
\draw (-0.5,1.3)--(0,1) (-0.5,0.7)--(0,1) (5,1)--(5.5,1.3) (5,1)--(5.5,1) (5,1)--(5.5,0.7);
\draw (-0.5,1.7)--(0,2) (5,2)--(5.5,1.7) (5,2)--(5.5,2.3);
\draw (-0.5,3)--(0,3) (-0.5,2.7)--(0,3) (-0.5,3.3)--(0,3) (5,3)--(5.5,3.3);
\draw (0,2) to [out=20, in=160] (5,2);
\draw (0,1)--(5,3) (0,3)--(5,1);
\draw [ultra thick] (1,1) to [out=70, in=200] (3,3);
\end{tikzpicture}}
&\raisebox{7mm}{$\rightarrow$}&
{\begin{tikzpicture}[scale=0.5]
\foreach \x in {0,1,2,3,4,5} {\foreach \y in {1,2,3} {\node at (\x,\y) [fill,circle,inner sep=1pt] {};};};
\draw (0,1)--(0,2)--(0,3) (1,1)--(1,2)--(1,3) (2,1)--(2,2)--(2,3) (3,1)--(3,2)--(3,3) (4,1)--(4,2)--(4,3) (5,1)--(5,2)--(5,3);
\draw (0,1)--(1,1)--(2,1)--(3,1)--(4,1)--(5,1);
\draw (0,2)--(1,2)--(2,2)--(3,2)--(4,2)--(5,2);
\draw (0,3)--(1,3)--(2,3)--(3,3)--(4,3)--(5,3);
\draw (-0.5,1.3)--(0,1) (-0.5,0.7)--(0,1) (5,1)--(5.5,1.3) (5,1)--(5.5,1) (5,1)--(5.5,0.7);
\draw (-0.5,1.7)--(0,2) (5,2)--(5.5,1.7) (5,2)--(5.5,2.3);
\draw (-0.5,3)--(0,3) (-0.5,2.7)--(0,3) (-0.5,3.3)--(0,3) (5,3)--(5.5,3.3);
\draw (0,2) to [out=20, in=160] (5,2);
\draw (0,1)--(5,3) (0,3)--(5,1);
\draw (1,1) to [out=70, in=200] (3,3);
\draw [ultra thick] (1,1) to [out=10,in=220] (4,2);
\end{tikzpicture}}
&\raisebox{7mm}{$\rightarrow$}&
{\begin{tikzpicture}[scale=0.5]
\foreach \x in {0,1,2,3,4,5} {\foreach \y in {1,2,3} {\node at (\x,\y) [fill,circle,inner sep=1pt] {};};};
\draw (0,1)--(0,2)--(0,3) (1,1)--(1,2)--(1,3) (2,1)--(2,2)--(2,3) (3,1)--(3,2)--(3,3) (4,1)--(4,2)--(4,3) (5,1)--(5,2)--(5,3);
\draw (1,1)--(2,1)--(3,1)--(4,1)--(5,1);
\draw (0,2)--(1,2)--(2,2)--(3,2)--(4,2)--(5,2);
\draw (0,3)--(1,3)--(2,3)--(3,3)--(4,3)--(5,3);
\draw (-0.5,1.3)--(0,1) (-0.5,0.7)--(0,1) (5,1)--(5.5,1.3) (5,1)--(5.5,1) (5,1)--(5.5,0.7);
\draw (-0.5,1.7)--(0,2) (5,2)--(5.5,1.7) (5,2)--(5.5,2.3);
\draw (-0.5,3)--(0,3) (-0.5,2.7)--(0,3) (-0.5,3.3)--(0,3) (5,3)--(5.5,3.3);
\draw (0,2) to [out=20, in=160] (5,2);
\draw (0,1)--(5,3) (0,3)--(5,1);
\draw (1,1) to [out=70, in=200] (3,3);
\draw (1,1) to [out=10,in=220] (4,2);
\draw [dashed] (0,1)--(1,1);
\end{tikzpicture}}
\\
\raisebox{7mm}{$\rightarrow$}&
{\begin{tikzpicture}[scale=0.5]
\foreach \x in {0,1,2,3,4,5} {\foreach \y in {1,2,3} {\node at (\x,\y) [fill,circle,inner sep=1pt] {};};};
\draw (0,1)--(0,2)--(0,3) (1,1)--(1,2)--(1,3) (2,1)--(2,2)--(2,3) (3,1)--(3,2)--(3,3) (4,1)--(4,2)--(4,3) (5,1)--(5,2)--(5,3);
\draw (0,1)--(0,2)--(0,3) (1,1)--(1,2)--(1,3) (2,1)--(2,2)--(2,3) (3,1)--(3,2)--(3,3) (4,1)--(4,2)--(4,3) (5,1)--(5,2)--(5,3);
\draw (1,1)--(2,1)--(3,1)--(4,1)--(5,1);
\draw (0,2)--(1,2)--(2,2)--(3,2)--(4,2)--(5,2);
\draw (0,3)--(1,3)--(2,3)--(3,3)--(4,3)--(5,3);
\draw (-0.5,1.3)--(0,1) (-0.5,0.7)--(0,1) (5,1)--(5.5,1.3) (5,1)--(5.5,1) (5,1)--(5.5,0.7);
\draw (-0.5,1.7)--(0,2) (5,2)--(5.5,1.7) (5,2)--(5.5,2.3);
\draw (-0.5,3)--(0,3) (-0.5,2.7)--(0,3) (-0.5,3.3)--(0,3) (5,3)--(5.5,3.3);
\draw (0,2) to [out=20, in=160] (5,2);
\draw (0,1)--(5,3) (0,3)--(5,1);
\draw (1,1) to [out=10,in=220] (4,2);
\draw [dashed] (1,1) to [out=70, in=200] (3,3);
\end{tikzpicture}}
&\raisebox{7mm}{$\rightarrow$}&
{\begin{tikzpicture}[scale=0.5]
\foreach \x in {0,1,2,3,4,5} {\foreach \y in {1,2,3} {\node at (\x,\y) [fill,circle,inner sep=1pt] {};};};
\draw (0,1)--(0,2)--(0,3) (1,1)--(1,2)--(1,3) (2,1)--(2,2)--(2,3) (3,1)--(3,2)--(3,3) (4,1)--(4,2)--(4,3) (5,1)--(5,2)--(5,3);
\draw (1,1)--(2,1)--(3,1)--(4,1)--(5,1);
\draw (0,2)--(1,2)--(2,2)--(3,2)--(4,2)--(5,2);
\draw (0,3)--(1,3)--(2,3)--(3,3)--(4,3)--(5,3);
\draw (-0.5,1.3)--(0,1) (-0.5,0.7)--(0,1) (5,1)--(5.5,1.3) (5,1)--(5.5,1) (5,1)--(5.5,0.7);
\draw (-0.5,1.7)--(0,2) (5,2)--(5.5,1.7) (5,2)--(5.5,2.3);
\draw (-0.5,3)--(0,3) (-0.5,2.7)--(0,3) (-0.5,3.3)--(0,3) (5,3)--(5.5,3.3);
\draw (0,2) to [out=20, in=160] (5,2);
\draw (0,1)--(5,3) (0,3)--(5,1);
\draw (1,1) to [out=10,in=220] (4,2);
\draw [ultra thick] (1,3) to [out=340, in=110] (3,1);
\end{tikzpicture}}
&\raisebox{7mm}{$\rightarrow$}&
{\begin{tikzpicture}[scale=0.5]
\foreach \x in {0,1,2,3,4,5} {\foreach \y in {1,2,3} {\node at (\x,\y) [fill,circle,inner sep=1pt] {};};};
\draw (0,1)--(0,2)--(0,3) (1,1)--(1,2)--(1,3) (2,1)--(2,2)--(2,3) (3,1)--(3,2)--(3,3) (4,1)--(4,2)--(4,3) (5,1)--(5,2)--(5,3);
\draw (1,1)--(2,1)--(3,1)--(4,1)--(5,1);
\draw (0,2)--(1,2)--(2,2)--(3,2)--(4,2)--(5,2);
\draw (0,3)--(1,3)--(2,3)--(3,3)--(4,3)--(5,3);
\draw (-0.5,1.3)--(0,1) (-0.5,0.7)--(0,1) (5,1)--(5.5,1.3) (5,1)--(5.5,1) (5,1)--(5.5,0.7);
\draw (-0.5,1.7)--(0,2) (5,2)--(5.5,1.7) (5,2)--(5.5,2.3);
\draw (-0.5,3)--(0,3) (-0.5,2.7)--(0,3) (-0.5,3.3)--(0,3) (5,3)--(5.5,3.3);
\draw (0,2) to [out=20, in=160] (5,2);
\draw (0,1)--(5,3) (0,3)--(5,1);
\draw (1,1) to [out=10,in=220] (4,2);
\draw (1,3) to [out=340, in=110] (3,1);
\draw [ultra thick] (1,3) to [out=350, in=140] (4,2);
\end{tikzpicture}}
\\
\raisebox{7mm}{$\rightarrow$}&
{\begin{tikzpicture}[scale=0.5]
\foreach \x in {0,1,2,3,4,5} {\foreach \y in {1,2,3} {\node at (\x,\y) [fill,circle,inner sep=1pt] {};};};
\draw (0,1)--(0,2)--(0,3) (1,1)--(1,2)--(1,3) (2,1)--(2,2)--(2,3) (3,1)--(3,2)--(3,3) (4,1)--(4,2)--(4,3) (5,1)--(5,2)--(5,3);
\draw (1,1)--(2,1)--(3,1)--(4,1)--(5,1);
\draw (0,2)--(1,2)--(2,2)--(3,2)--(4,2)--(5,2);
\draw (1,3)--(2,3)--(3,3)--(4,3)--(5,3);
\draw (-0.5,1.3)--(0,1) (-0.5,0.7)--(0,1) (5,1)--(5.5,1.3) (5,1)--(5.5,1) (5,1)--(5.5,0.7);
\draw (-0.5,1.7)--(0,2) (5,2)--(5.5,1.7) (5,2)--(5.5,2.3);
\draw (-0.5,3)--(0,3) (-0.5,2.7)--(0,3) (-0.5,3.3)--(0,3) (5,3)--(5.5,3.3);
\draw (0,2) to [out=20, in=160] (5,2);
\draw (0,1)--(5,3) (0,3)--(5,1);
\draw (1,1) to [out=10,in=220] (4,2);
\draw (1,3) to [out=340, in=110] (3,1);
\draw (1,3) to [out=350, in=140] (4,2);
\draw [dashed] (0,3)--(1,3);
\end{tikzpicture}}
&\raisebox{7mm}{$\rightarrow$}&
{\begin{tikzpicture}[scale=0.5]
\foreach \x in {0,1,2,3,4,5} {\foreach \y in {1,2,3} {\node at (\x,\y) [fill,circle,inner sep=1pt] {};};};
\draw (0,1)--(0,2)--(0,3) (1,1)--(1,2)--(1,3) (2,1)--(2,2)--(2,3) (3,1)--(3,2)--(3,3) (4,1)--(4,2)--(4,3) (5,1)--(5,2)--(5,3);
\draw (1,1)--(2,1)--(3,1)--(4,1)--(5,1);
\draw (0,2)--(1,2)--(2,2)--(3,2)--(4,2)--(5,2);
\draw (1,3)--(2,3)--(3,3)--(4,3)--(5,3);
\draw (-0.5,1.3)--(0,1) (-0.5,0.7)--(0,1) (5,1)--(5.5,1.3) (5,1)--(5.5,1) (5,1)--(5.5,0.7);
\draw (-0.5,1.7)--(0,2) (5,2)--(5.5,1.7) (5,2)--(5.5,2.3);
\draw (-0.5,3)--(0,3) (-0.5,2.7)--(0,3) (-0.5,3.3)--(0,3) (5,3)--(5.5,3.3);
\draw (0,2) to [out=20, in=160] (5,2);
\draw (0,1)--(5,3) (0,3)--(5,1);
\draw (1,1) to [out=10,in=220] (4,2);
\draw (1,3) to [out=350, in=140] (4,2);
\draw [dashed] (1,3) to [out=340, in=110] (3,1);
\end{tikzpicture}}
&\raisebox{7mm}{$\rightarrow$}&
{\begin{tikzpicture}[scale=0.5]
\foreach \x in {0,1,2,3,4,5} {\foreach \y in {1,2,3} {\node at (\x,\y) [fill,circle,inner sep=1pt] {};};};
\draw (0,1)--(0,2)--(0,3) (1,1)--(1,2)--(1,3) (2,1)--(2,2)--(2,3) (3,1)--(3,2)--(3,3) (4,1)--(4,2)--(4,3) (5,1)--(5,2)--(5,3);
\draw (1,1)--(2,1)--(3,1)--(4,1)--(5,1);
\draw (0,2)--(1,2)--(2,2)--(3,2)--(4,2)--(5,2);
\draw (1,3)--(2,3)--(3,3)--(4,3)--(5,3);
\draw (-0.5,1.3)--(0,1) (-0.5,0.7)--(0,1) (5,1)--(5.5,1.3) (5,1)--(5.5,1) (5,1)--(5.5,0.7);
\draw (-0.5,1.7)--(0,2) (5,2)--(5.5,1.7) (5,2)--(5.5,2.3);
\draw (-0.5,3)--(0,3) (-0.5,2.7)--(0,3) (-0.5,3.3)--(0,3) (5,3)--(5.5,3.3);
\draw (0,2) to [out=20, in=160] (5,2);
\draw (0,1)--(5,3) (0,3)--(5,1);
\draw (1,1) to [out=10,in=220] (4,2);
\draw (1,3) to [out=350, in=140] (4,2);
\draw [ultra thick] (1,2) to [out=40, in=190] (4,3);
\end{tikzpicture}}
\\
\raisebox{7mm}{$\rightarrow$}&
{\begin{tikzpicture}[scale=0.5]
\foreach \x in {0,1,2,3,4,5} {\foreach \y in {1,2,3} {\node at (\x,\y) [fill,circle,inner sep=1pt] {};};};
\draw (0,1)--(0,2)--(0,3) (1,1)--(1,2)--(1,3) (2,1)--(2,2)--(2,3) (3,1)--(3,2)--(3,3) (4,1)--(4,2)--(4,3) (5,1)--(5,2)--(5,3);
\draw (1,1)--(2,1)--(3,1)--(4,1)--(5,1);
\draw (0,2)--(1,2)--(2,2)--(3,2)--(4,2)--(5,2);
\draw (1,3)--(2,3)--(3,3)--(4,3)--(5,3);
\draw (-0.5,1.3)--(0,1) (-0.5,0.7)--(0,1) (5,1)--(5.5,1.3) (5,1)--(5.5,1) (5,1)--(5.5,0.7);
\draw (-0.5,1.7)--(0,2) (5,2)--(5.5,1.7) (5,2)--(5.5,2.3);
\draw (-0.5,3)--(0,3) (-0.5,2.7)--(0,3) (-0.5,3.3)--(0,3) (5,3)--(5.5,3.3);
\draw (0,2) to [out=20, in=160] (5,2);
\draw (0,1)--(5,3) (0,3)--(5,1);
\draw (1,1) to [out=10,in=220] (4,2);
\draw (1,3) to [out=350, in=140] (4,2);
\draw (1,2) to [out=40, in=190] (4,3);
\draw [ultra thick] (1,2) to [out=20, in=160] (3,2);
\end{tikzpicture}}
&\raisebox{7mm}{$\rightarrow$}&
{\begin{tikzpicture}[scale=0.5]
\foreach \x in {0,1,2,3,4,5} {\foreach \y in {1,2,3} {\node at (\x,\y) [fill,circle,inner sep=1pt] {};};};
\draw (0,1)--(0,2)--(0,3) (1,1)--(1,2)--(1,3) (2,1)--(2,2)--(2,3) (3,1)--(3,2)--(3,3) (4,1)--(4,2)--(4,3) (5,1)--(5,2)--(5,3);
\draw (1,1)--(2,1)--(3,1)--(4,1)--(5,1);
\draw (0,2)--(1,2)--(2,2)--(3,2)--(4,2)--(5,2);
\draw (1,3)--(2,3)--(3,3)--(4,3)--(5,3);
\draw (-0.5,1.3)--(0,1) (-0.5,0.7)--(0,1) (5,1)--(5.5,1.3) (5,1)--(5.5,1) (5,1)--(5.5,0.7);
\draw (-0.5,1.7)--(0,2) (5,2)--(5.5,1.7) (5,2)--(5.5,2.3);
\draw (-0.5,3)--(0,3) (-0.5,2.7)--(0,3) (-0.5,3.3)--(0,3) (5,3)--(5.5,3.3);
\draw (0,2) to [out=20, in=160] (5,2);
\draw (0,1)--(5,3) (0,3)--(5,1);
\draw (1,1) to [out=10,in=220] (4,2);
\draw (1,3) to [out=350, in=140] (4,2);
\draw (1,2) to [out=40, in=190] (4,3);
\draw (1,2) to [out=20, in=160] (3,2);
\draw [ultra thick] (1,2) to [out=320, in=170] (4,1);
\end{tikzpicture}}
&\raisebox{7mm}{$\rightarrow$}&
{\begin{tikzpicture}[scale=0.5]
\foreach \x in {0,1,2,3,4,5} {\foreach \y in {1,2,3} {\node at (\x,\y) [fill,circle,inner sep=1pt] {};};};
\draw (0,1)--(0,2)--(0,3) (1,1)--(1,2)--(1,3) (2,1)--(2,2)--(2,3) (3,1)--(3,2)--(3,3) (4,1)--(4,2)--(4,3) (5,1)--(5,2)--(5,3);
\draw (1,1)--(2,1)--(3,1)--(4,1)--(5,1);
\draw (1,2)--(2,2)--(3,2)--(4,2)--(5,2);
\draw (1,3)--(2,3)--(3,3)--(4,3)--(5,3);
\draw (-0.5,1.3)--(0,1) (-0.5,0.7)--(0,1) (5,1)--(5.5,1.3) (5,1)--(5.5,1) (5,1)--(5.5,0.7);
\draw (-0.5,1.7)--(0,2) (5,2)--(5.5,1.7) (5,2)--(5.5,2.3);
\draw (-0.5,3)--(0,3) (-0.5,2.7)--(0,3) (-0.5,3.3)--(0,3) (5,3)--(5.5,3.3);
\draw (0,2) to [out=20, in=160] (5,2);
\draw (0,1)--(5,3) (0,3)--(5,1);
\draw (1,1) to [out=10,in=220] (4,2);
\draw (1,3) to [out=350, in=140] (4,2);
\draw (1,2) to [out=40, in=190] (4,3);
\draw (1,2) to [out=20, in=160] (3,2);
\draw (1,2) to [out=320, in=170] (4,1);
\draw [dashed] (0,2)--(1,2);
\end{tikzpicture}}
\end{tabular}
\caption{A sequence of operations which induces $I(H) \simpsimeq I(G \sqcup C_{1,8})$ in the proof of Theorem \ref{mainm3} (Step 2)}
\label{mainm3step2}
\end{figure}

\item[Step 3] Delete three vertices, $\overline{5}$, $5$ and $\widehat{5}$. This process is performed as follows:
\begin{align*}
&\del(\overline{2}\overline{4},3), \add(\overline{3}\overline{5},4), \del(\overline{5},\overline{2}), \add(\widehat{3}5,\overline{4}), \del(5,\widehat{2}),\\
&\add(3\widehat{5},\overline{4}), \del(\widehat{5},2)
\end{align*}
(see Figure \ref{mainm3step34}).

\item[Step 4] Delete the vertex $\overline{3}$ by $\del(\overline{3}, 2)$ (see Figure \ref{mainm3step34}).
\begin{figure}[H]
\centering
\begin{tabular}{cccccc}
\raisebox{7mm}{$\rightarrow$}&
{\begin{tikzpicture}[scale=0.5]
\foreach \x in {0,1,2,3,4,5} {\foreach \y in {1,2,3} {\node at (\x,\y) [fill,circle,inner sep=1pt] {};};};
\draw (0,1)--(0,2)--(0,3) (1,1)--(1,2)--(1,3) (2,1)--(2,2)--(2,3) (3,1)--(3,2)--(3,3) (4,1)--(4,2)--(4,3) (5,1)--(5,2)--(5,3);
\draw (1,1)--(2,1)--(3,1)--(4,1)--(5,1);
\draw (1,2)--(2,2)--(3,2)--(4,2)--(5,2);
\draw (1,3)--(2,3)--(3,3)--(4,3)--(5,3);
\draw (-0.5,1.3)--(0,1) (-0.5,0.7)--(0,1) (5,1)--(5.5,1.3) (5,1)--(5.5,1) (5,1)--(5.5,0.7);
\draw (-0.5,1.7)--(0,2) (5,2)--(5.5,1.7) (5,2)--(5.5,2.3);
\draw (-0.5,3)--(0,3) (-0.5,2.7)--(0,3) (-0.5,3.3)--(0,3) (5,3)--(5.5,3.3);
\draw (0,2) to [out=20, in=160] (5,2);
\draw (0,1)--(5,3) (0,3)--(5,1);
\draw (1,1) to [out=10,in=220] (4,2);
\draw (1,3) to [out=350, in=140] (4,2);
\draw (1,2) to [out=40, in=190] (4,3);
\draw (1,2) to [out=320, in=170] (4,1);
\draw [dashed] (1,2) to [out=20, in=160] (3,2);
\end{tikzpicture}}
&\raisebox{7mm}{$\rightarrow$}&
{\begin{tikzpicture}[scale=0.5]
\foreach \x in {0,1,2,3,4,5} {\foreach \y in {1,2,3} {\node at (\x,\y) [fill,circle,inner sep=1pt] {};};};
\draw (0,1)--(0,2)--(0,3) (1,1)--(1,2)--(1,3) (2,1)--(2,2)--(2,3) (3,1)--(3,2)--(3,3) (4,1)--(4,2)--(4,3) (5,1)--(5,2)--(5,3);
\draw (1,1)--(2,1)--(3,1)--(4,1)--(5,1);
\draw (1,2)--(2,2)--(3,2)--(4,2)--(5,2);
\draw (1,3)--(2,3)--(3,3)--(4,3)--(5,3);
\draw (-0.5,1.3)--(0,1) (-0.5,0.7)--(0,1) (5,1)--(5.5,1.3) (5,1)--(5.5,1) (5,1)--(5.5,0.7);
\draw (-0.5,1.7)--(0,2) (5,2)--(5.5,1.7) (5,2)--(5.5,2.3);
\draw (-0.5,3)--(0,3) (-0.5,2.7)--(0,3) (-0.5,3.3)--(0,3) (5,3)--(5.5,3.3);
\draw (0,2) to [out=20, in=160] (5,2);
\draw (0,1)--(5,3) (0,3)--(5,1);
\draw (1,1) to [out=10,in=220] (4,2);
\draw (1,3) to [out=350, in=140] (4,2);
\draw (1,2) to [out=40, in=190] (4,3);
\draw (1,2) to [out=320, in=170] (4,1);
\draw [ultra thick] (2,2) to [out=20, in=160] (4,2);
\end{tikzpicture}}
&\raisebox{7mm}{$\rightarrow$}&
{\begin{tikzpicture}[scale=0.5]
\foreach \x in {0,1,2,3,5} {\foreach \y in {1,2,3} {\node at (\x,\y) [fill,circle,inner sep=1pt] {};};};
\node at (4,1) [fill,circle,inner sep=1pt] {};
\node at (4,3) [fill,circle,inner sep=1pt] {};
\draw (0,1)--(0,2)--(0,3) (1,1)--(1,2)--(1,3) (2,1)--(2,2)--(2,3) (3,1)--(3,2)--(3,3) (5,1)--(5,2)--(5,3);
\draw (1,1)--(2,1)--(3,1)--(4,1)--(5,1);
\draw (1,2)--(2,2)--(3,2);
\draw (1,3)--(2,3)--(3,3)--(4,3)--(5,3);
\draw (-0.5,1.3)--(0,1) (-0.5,0.7)--(0,1) (5,1)--(5.5,1.3) (5,1)--(5.5,1) (5,1)--(5.5,0.7);
\draw (-0.5,1.7)--(0,2) (5,2)--(5.5,1.7) (5,2)--(5.5,2.3);
\draw (-0.5,3)--(0,3) (-0.5,2.7)--(0,3) (-0.5,3.3)--(0,3) (5,3)--(5.5,3.3);
\draw (0,2) to [out=20, in=160] (5,2);
\draw (0,1)--(5,3) (0,3)--(5,1);
\draw (1,2) to [out=40, in=190] (4,3);
\draw (1,2) to [out=320, in=170] (4,1);
\draw [dashed] (3,2)--(5,2) (4,1)--(4,3);
\draw [dashed] (2,2) to [out=20, in=160] (4,2);
\draw [dashed] (1,1) to [out=10,in=220] (4,2);
\draw [dashed] (1,3) to [out=350, in=140] (4,2);
\node at (4,2) [circle, draw=black, fill=white, inner sep=1pt] {};
\end{tikzpicture}}
\\
\raisebox{7mm}{$\rightarrow$}&
{\begin{tikzpicture}[scale=0.5]
\foreach \x in {0,1,2,3,5} {\foreach \y in {1,2,3} {\node at (\x,\y) [fill,circle,inner sep=1pt] {};};};
\node at (4,1) [fill,circle,inner sep=1pt] {};
\node at (4,3) [fill,circle,inner sep=1pt] {};
\draw (0,1)--(0,2)--(0,3) (1,1)--(1,2)--(1,3) (2,1)--(2,2)--(2,3) (3,1)--(3,2)--(3,3) (5,1)--(5,2)--(5,3);
\draw (1,1)--(2,1)--(3,1)--(4,1)--(5,1);
\draw (1,2)--(2,2)--(3,2);
\draw (1,3)--(2,3)--(3,3)--(4,3)--(5,3);
\draw (-0.5,1.3)--(0,1) (-0.5,0.7)--(0,1) (5,1)--(5.5,1.3) (5,1)--(5.5,1) (5,1)--(5.5,0.7);
\draw (-0.5,1.7)--(0,2) (5,2)--(5.5,1.7) (5,2)--(5.5,2.3);
\draw (-0.5,3)--(0,3) (-0.5,2.7)--(0,3) (-0.5,3.3)--(0,3) (5,3)--(5.5,3.3);
\draw (0,2) to [out=20, in=160] (5,2);
\draw (0,1)--(5,3) (0,3)--(5,1);
\draw (1,2) to [out=40, in=190] (4,3);
\draw (1,2) to [out=320, in=170] (4,1);
\draw [ultra thick] (2,3) to [out=340,in=110] (4,1);
\end{tikzpicture}}
&\raisebox{7mm}{$\rightarrow$}&
{\begin{tikzpicture}[scale=0.5]
\foreach \x in {0,1,2,3,5} {\foreach \y in {1,2,3} {\node at (\x,\y) [fill,circle,inner sep=1pt] {};};};
\node at (4,3) [fill,circle,inner sep=1pt] {};
\draw (0,1)--(0,2)--(0,3) (1,1)--(1,2)--(1,3) (2,1)--(2,2)--(2,3) (3,1)--(3,2)--(3,3) (5,1)--(5,2)--(5,3);
\draw (1,1)--(2,1)--(3,1);
\draw (1,2)--(2,2)--(3,2);
\draw (1,3)--(2,3)--(3,3)--(4,3)--(5,3);
\draw (-0.5,1.3)--(0,1) (-0.5,0.7)--(0,1) (5,1)--(5.5,1.3) (5,1)--(5.5,1) (5,1)--(5.5,0.7);
\draw (-0.5,1.7)--(0,2) (5,2)--(5.5,1.7) (5,2)--(5.5,2.3);
\draw (-0.5,3)--(0,3) (-0.5,2.7)--(0,3) (-0.5,3.3)--(0,3) (5,3)--(5.5,3.3);
\draw (0,2) to [out=20, in=160] (5,2);
\draw (0,1)--(5,3) (0,3)--(5,1);
\draw (1,2) to [out=40, in=190] (4,3);
\draw [dashed] (3,1)--(5,1);
\draw [dashed] (2,3) to [out=340,in=110] (4,1);
\draw [dashed] (1,2) to [out=320, in=170] (4,1);
\node at (4,1) [circle, draw=black, fill=white, inner sep=1pt] {};
\end{tikzpicture}}
&\raisebox{7mm}{$\rightarrow$}&
{\begin{tikzpicture}[scale=0.5]
\foreach \x in {0,1,2,3,5} {\foreach \y in {1,2,3} {\node at (\x,\y) [fill,circle,inner sep=1pt] {};};};
\node at (4,3) [fill,circle,inner sep=1pt] {};
\draw (0,1)--(0,2)--(0,3) (1,1)--(1,2)--(1,3) (2,1)--(2,2)--(2,3) (3,1)--(3,2)--(3,3) (5,1)--(5,2)--(5,3);
\draw (1,1)--(2,1)--(3,1);
\draw (1,2)--(2,2)--(3,2);
\draw (1,3)--(2,3)--(3,3)--(4,3)--(5,3);
\draw (-0.5,1.3)--(0,1) (-0.5,0.7)--(0,1) (5,1)--(5.5,1.3) (5,1)--(5.5,1) (5,1)--(5.5,0.7);
\draw (-0.5,1.7)--(0,2) (5,2)--(5.5,1.7) (5,2)--(5.5,2.3);
\draw (-0.5,3)--(0,3) (-0.5,2.7)--(0,3) (-0.5,3.3)--(0,3) (5,3)--(5.5,3.3);
\draw (0,2) to [out=20, in=160] (5,2);
\draw (0,1)--(5,3) (0,3)--(5,1);
\draw (1,2) to [out=40, in=190] (4,3);
\draw [ultra thick] (2,1) to [out=20, in=250] (4,3);
\end{tikzpicture}}
\\
\raisebox{7mm}{$\rightarrow$}&
{\begin{tikzpicture}[scale=0.5]
\foreach \x in {0,1,2,3,5} {\foreach \y in {1,2,3} {\node at (\x,\y) [fill,circle,inner sep=1pt] {};};};
\draw (0,1)--(0,2)--(0,3) (1,1)--(1,2)--(1,3) (2,1)--(2,2)--(2,3) (3,1)--(3,2)--(3,3) (5,1)--(5,2)--(5,3);
\draw (1,1)--(2,1)--(3,1);
\draw (1,2)--(2,2)--(3,2);
\draw (1,3)--(2,3)--(3,3);
\draw (-0.5,1.3)--(0,1) (-0.5,0.7)--(0,1) (5,1)--(5.5,1.3) (5,1)--(5.5,1) (5,1)--(5.5,0.7);
\draw (-0.5,1.7)--(0,2) (5,2)--(5.5,1.7) (5,2)--(5.5,2.3);
\draw (-0.5,3)--(0,3) (-0.5,2.7)--(0,3) (-0.5,3.3)--(0,3) (5,3)--(5.5,3.3);
\draw (0,2) to [out=20, in=160] (5,2);
\draw (0,1)--(5,3) (0,3)--(5,1);
\draw [dashed] (1,2) to [out=40, in=190] (4,3);
\draw [dashed] (2,1) to [out=20, in=250] (4,3);
\draw [dashed] (3,3)--(5,3);
\node at (4,3) [circle, draw=black, fill=white, inner sep=1pt] {};
\end{tikzpicture}}
&\raisebox{7mm}{$\rightarrow$}&
{\begin{tikzpicture}[scale=0.5]
\foreach \x in {0,1,3,5} {\foreach \y in {1,2,3} {\node at (\x,\y) [fill,circle,inner sep=1pt] {};};};
\node at (2,1) [fill,circle,inner sep=1pt] {};
\node at (2,3) [fill,circle,inner sep=1pt] {};
\draw (0,1)--(0,2)--(0,3) (1,1)--(1,2)--(1,3) (3,1)--(3,2)--(3,3) (5,1)--(5,2)--(5,3);
\draw (1,1)--(2,1)--(3,1);
\draw (1,3)--(2,3)--(3,3);
\draw (-0.5,1.3)--(0,1) (-0.5,0.7)--(0,1) (5,1)--(5.5,1.3) (5,1)--(5.5,1) (5,1)--(5.5,0.7);
\draw (-0.5,1.7)--(0,2) (5,2)--(5.5,1.7) (5,2)--(5.5,2.3);
\draw (-0.5,3)--(0,3) (-0.5,2.7)--(0,3) (-0.5,3.3)--(0,3) (5,3)--(5.5,3.3);
\draw (0,2) to [out=20, in=160] (5,2);
\draw (0,1)--(5,3) (0,3)--(5,1);
\draw [dashed] (1,2)--(3,2) (2,1)--(2,3);
\node at (2,2) [circle, draw=black, fill=white, inner sep=1pt] {};
\end{tikzpicture}}
\end{tabular}
\caption{A sequence of operations which induces $I(H) \simpsimeq I(G \sqcup C_{1,8})$ in the proof of Theorem \ref{mainm3} (Step 3, Step 4)}
\label{mainm3step34}
\end{figure}
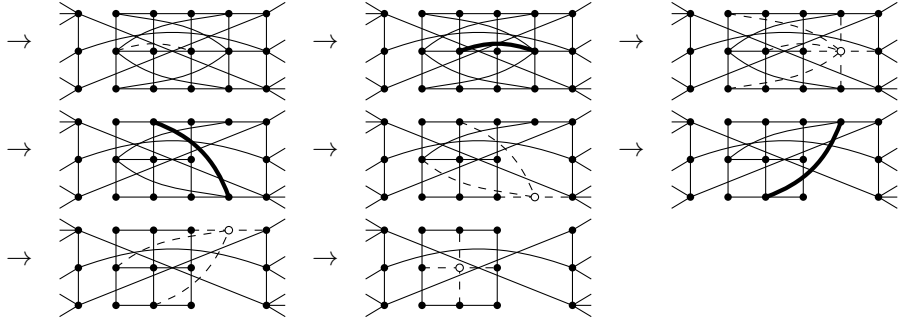
\end{description}
\end{proof}

\begin{remark}
Let $(m,k) = (1,5), (2,4), (3,6)$. In the proofs of Theorem \ref{mainm1}, Theorem \ref{mainm2} and Theorem \ref{mainm3}, we obtained $G \sqcup P_{m,k-3}$ from $H$ without changing the simple homotopy types of the independence complexes by adding $m$ edges $(i,1)(m-i+1,k)$, deleting edges $(i,1)(i,2)$ and deleting vertices $(i,k-1)$ of $P_{m,k}$ for $i=1,2, \ldots, m$ (see Figure \ref{pairmk}).
\begin{figure}[H]
\centering
\begin{tabular}{ccc}
{\begin{tikzpicture}
[scale=0.6]
\foreach \x in {1,2,3,4,5} {\foreach \y in {1} {\node at (\x,\y) [fill,circle,inner sep=1.0pt] {};};};
\draw (1,1)--(5,1);
\node at (3,2) {$H$};
\draw (0.5,0.7)--(1,1) (0.5,1.3)--(1,1) (5,1)--(5.5,0.7) (5,1)--(5.5,1) (5,1)--(5.5,1.3);
\end{tikzpicture}}
&{\begin{tikzpicture}[scale=0.6]\node at (0,2.3) {};\node at (0,0.7) {};\node at (0,1) {$\longrightarrow$};\end{tikzpicture}}&
{\begin{tikzpicture}
[scale=0.6]
\foreach \x in {1,2,3,5} {\foreach \y in {1} {\node at (\x,\y) [fill,circle,inner sep=1.0pt] {};};};
\draw (2,1)--(3,1);
\node at (3,2) {$G \sqcup P_{1,2}$};
\draw (1,1) to [out=20, in=160] (5,1);
\draw (0.5,0.7)--(1,1) (0.5,1.3)--(1,1) (5,1)--(5.5,0.7) (5,1)--(5.5,1) (5,1)--(5.5,1.3);
\end{tikzpicture}}\\
{\begin{tikzpicture}
[scale=0.6]
\foreach \x in {1,2,3,4} {\foreach \y in {1,2} {\node at (\x,\y) [fill,circle,inner sep=1.0pt] {};};};
\draw (1,1)--(2,1)--(3,1)--(4,1);
\draw (1,2)--(2,2)--(3,2)--(4,2);
\draw (1,1)--(1,2) (2,1)--(2,2) (3,1)--(3,2) (4,1)--(4,2);
\draw (0.5,1.7)--(1,2) (0.5,2.3)--(1,2) (0.5,0.7)--(1,1) (4,1)--(4.5,0.7) (4,1)--(4.5,1) (4,1)--(4.5,1.3) (4,2)--(4.5,1.7);
\node at (2.5,2.5) {$H$};
\end{tikzpicture}}
&{\begin{tikzpicture}[scale=0.6]\node at (0,2.3) {};\node at (0,0.7) {};\node at (0,1.5) {$\longrightarrow$};\end{tikzpicture}}&
{\begin{tikzpicture}
[scale=0.6]
\foreach \x in {1,2,4} {\foreach \y in {1,2} {\node at (\x,\y) [fill,circle,inner sep=1.0pt] {};};};
\draw (1,1)--(4,2) (1,2)--(4,1);
\draw (1,1)--(1,2) (4,1)--(4,2);
\draw (0.5,1.7)--(1,2) (0.5,2.3)--(1,2) (0.5,0.7)--(1,1) (4,1)--(4.5,0.7) (4,1)--(4.5,1) (4,1)--(4.5,1.3) (4,2)--(4.5,1.7);
\draw (2,1)--(2,2);
\node at (2.5,2.5) {$G \sqcup P_{2,1}$};
\end{tikzpicture}}\\
{\begin{tikzpicture}
[scale=0.6]
\foreach \x in {0,1,2,3,4,5} {\foreach \y in {1,2,3} {\node at (\x,\y) [fill,circle,inner sep=1.0pt] {};};};
\draw (0,1)--(0,2)--(0,3) (1,1)--(1,2)--(1,3) (2,1)--(2,2)--(2,3) (3,1)--(3,2)--(3,3) (4,1)--(4,2)--(4,3) (5,1)--(5,2)--(5,3);
\draw (0,1)--(1,1)--(2,1)--(3,1)--(4,1)--(5,1);
\draw (0,2)--(1,2)--(2,2)--(3,2)--(4,2)--(5,2);
\draw (0,3)--(1,3)--(2,3)--(3,3)--(4,3)--(5,3);
\draw (-0.5,1.3)--(0,1) (-0.5,0.7)--(0,1) (5,1)--(5.5,1.3) (5,1)--(5.5,1) (5,1)--(5.5,0.7);
\draw (-0.5,1.7)--(0,2) (5,2)--(5.5,1.7) (5,2)--(5.5,2.3);
\draw (-0.5,3)--(0,3) (-0.5,2.7)--(0,3) (-0.5,3.3)--(0,3) (5,3)--(5.5,3.3);
\node at (2.5,3.5) {$H$};
\end{tikzpicture}}
&{\begin{tikzpicture}[scale=0.6]\node at (0,3.3) {};\node at (0,0.7) {};\node at (0,2) {$\longrightarrow$};\end{tikzpicture}}&
{\begin{tikzpicture}[scale=0.6]
\foreach \x in {0,1,2,3,5} {\foreach \y in {1,2,3} {\node at (\x,\y) [fill,circle,inner sep=1.0pt] {};};};
\draw (0,1)--(0,2)--(0,3) (5,1)--(5,2)--(5,3);
\draw (-0.5,1.3)--(0,1) (5,1)--(5.5,1.3) (-0.5,0.7)--(0,1) (5,1)--(5.5,1) (5,1)--(5.5,0.7);
\draw (-0.5,1.7)--(0,2) (5,2)--(5.5,1.7) (5,2)--(5.5,2.3);
\draw (-0.5,3)--(0,3) (5,3)--(5.5,3.3) (-0.5,2.7)--(0,3) (-0.5,3.3)--(0,3);
\draw (0,2) to [out=20, in=160] (5,2);
\draw (0,1)--(5,3) (0,3)--(5,1);
\draw (1,1) grid (3,3);
\node at (2.5,3.5) {$G \sqcup P_{3,3}$};
\end{tikzpicture}}
\end{tabular}
\caption{The replacements in the proofs of Theorem \ref{mainm1}, Theorem \ref{mainm2} and Theorem \ref{mainm3}}
\label{pairmk}
\end{figure}
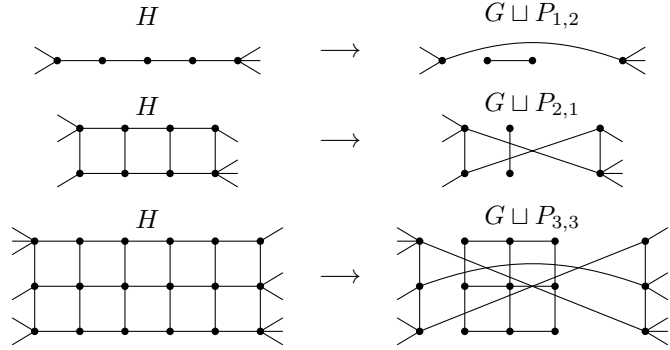

However, this does not hold for $m=4$. Assume that there exists a natural number $k$ such that
\begin{align*}
I(P_{4,k}) \simpsimeq I(P_{4,k-3} \sqcup P_{4,2})
\end{align*}
(see Figure \ref{P4kfigure}).
\begin{figure}[H]
\centering
\begin{tabular}{ccc}
{\begin{tikzpicture}[scale=0.6]
\foreach \x in {1,2,3,4,5,6,7} {\foreach \y in {1,2,3,4} {\node at (\x,\y) [fill,circle,inner sep=1.0pt] {};};};
\draw (1,1) grid (7,4);
\node at (4,4.5) {$P_{4,k}$};
\end{tikzpicture}}
&{\begin{tikzpicture}[scale=0.6]\node at (0,3.3) {};\node at (0,0.7) {};\node at (0,2) {$\longrightarrow$};\end{tikzpicture}}&
{\begin{tikzpicture}[scale=0.6]
\foreach \x in {1,2,3,4,5,7} {\foreach \y in {1,2,3,4} {\node at (\x,\y) [fill,circle,inner sep=1.0pt] {};};};
\draw (1,1)--(1,4) (7,1)--(7,4) (2,1) grid (5,4);
\draw (1,1)--(7,4) (1,4)--(7,1) (1,2)--(7,3) (1,3)--(7,2);
\node at (4,4.5) {$P_{4, k-3} \sqcup P_{4,2}$};
\end{tikzpicture}}
\end{tabular}
\caption{$P_{4,k}$ and $P_{4, k-3} \sqcup P_{4,2}$}
\label{P4kfigure}
\end{figure}
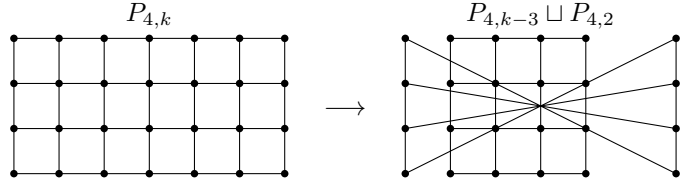

We have $I(P_{4,2}) \searrow {\dia}^1$ by 
$$\del(\overline{1}, 2), \del(\overline{2},1), \del(\widehat{1}, \widetilde{2}), \del(\widehat{2}, \widetilde{1}).$$
Thus, we have $I(P_{4,k}) \simpsimeq \Sigma^2 I(P_{4,k-3})$, which implies $\widetilde{\chi}(I(P_{4,k})) = \widetilde{\chi}(I(P_{4,k-3}))$, where $\widetilde{\chi}(K)$ denotes the reduced Euler characteristic of a simplicial complex $K$.  On the other hand, we have
\begin{align*}
&\widetilde{\chi}(I(P_{4,6l+i})) = \left \{
\begin{aligned}
&-2l-1 &\quad &(i = 0, 2), \\
&2l &\quad &(i = 1), \\
&2l+1 &\quad &(i = 3,5), \\
&-2l-2 &\quad &(i=4) ,
\end{aligned} \right. 
\end{align*}
which we will prove in Appendix \ref{appendixA}. In particular, we have $\widetilde{\chi}(I(P_{4,n})) \geq 0$ when $n$ is odd and $\widetilde{\chi}(I(P_{4,n})) <0$ when $n$ is even. 
This is a contradiction to $\widetilde{\chi}(I(P_{4,k})) = \widetilde{\chi}(I(P_{4,k-3}))$.

\end{remark}

\begin{proof}[Proof of Corollary \ref{CnnCnnnsimple} and Corollary \ref{MnnMnnnsimple}]
By Theorem \ref{mainm2} and Theorem \ref{mainm3}, we have
\begin{align*}
&I(C_{2,n+2}) \simpsimeq \Sigma I(M_{2,n}),& &I(M_{2,n+2}) \simpsimeq \Sigma I(M_{2,n}), \\
&I(C_{3,n+4}) \simpsimeq \Sigma^3 I(M_{3,n}),& &I(M_{3,n+4}) \simpsimeq \Sigma^3 I(M_{3,n}).
\end{align*}
The base cases for determining $I(C_{2,n})$ and $I(M_{2,n})$ are
\begin{align*}
&I(C_{2,1}) = {\dia}^{-1},& &I(C_{2,2}) \searrow {\dia}^0,&
&I(M_{2,1}) = {\dia}^0,& &I(M_{2,2}) = {\bigvee}_3 {\dia}^0 
\end{align*} 
(see Figure \ref{mainm2basecases}).
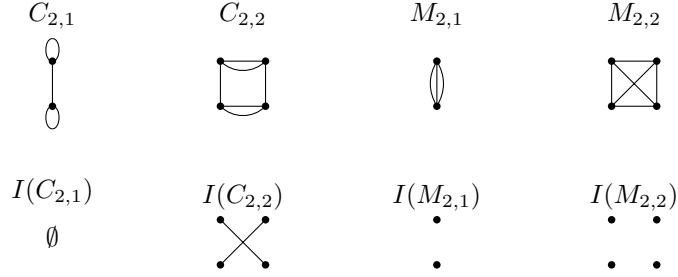
\begin{figure}[H]
\centering
\begin{tabular}{ccccccc}
\begin{tikzpicture}
[scale=0.6]
\draw (0,2) to [out=0, in=0] (0,1.5); 
\draw (0,1.5) to [out=180, in=180] (0,2);
\draw (0,3) to [out=0, in=0] (0,3.5); 
\draw (0,3.5) to [out=180, in=180] (0,3);
\draw (0,2)--(0,3);
\node at (0,2) [fill, circle, black, inner sep=1pt] {};
\node at (0,3) [fill, circle, black, inner sep=1pt] {};
\node at (0,4) {$C_{2,1}$};
\node at (0,1) {};
\end{tikzpicture}& \hspace{0.2in} &
\begin{tikzpicture}
[scale=0.6]
\foreach \x in {(0,2),(1,2),(0,3),(1,3)} {\node at \x [fill, circle, black, inner sep=1pt] {};};
\draw (0,2)--(1,2)--(1,3)--(0,3)--cycle;
\draw (0,2) to [out=-45, in=-135] (1,2);
\draw (0,3) to [out=-45, in=-135] (1,3);
\node at (0.5,4) {$C_{2,2}$};
\node at (0,1) {};
\end{tikzpicture}& \hspace{0.2in} & 
\begin{tikzpicture}
[scale=0.6]
\draw (0,2)--(0,3);
\node at (0,2) [fill, circle, black, inner sep=1pt] {};
\node at (0,3) [fill, circle, black, inner sep=1pt] {};
\draw (0,2) to [out=60, in=-60] (0,3);
\draw (0,2) to [out=120, in=-120] (0,3);
\node at (0,4) {$M_{2,1}$};
\node at (0,1) {};
\end{tikzpicture}& \hspace{0.2in} &
\begin{tikzpicture}
[scale=0.6]
\draw (0,2)--(0,3)--(1,3)--(1,2)--cycle (0,2)--(1,3) (0,3)--(1,2);
\foreach \x in {(0,2),(1,2),(0,3),(1,3)} {\node at \x [fill, circle, black, inner sep=1pt] {};};
\node at (0.5,4) {$M_{2,2}$};
\node at (0,1) {};
\end{tikzpicture}
\\
\begin{tikzpicture}
[scale=0.6]
\node at (0,0.5) {$\emptyset$};
\node at (0,1.5) {$I(C_{2,1})$};
\node at (0,0) {};
\end{tikzpicture}& \hspace{0.2in} &
\begin{tikzpicture}
[scale=0.6]
\foreach \x in {(0,0),(1,0),(0,1),(1,1)} {\node at \x [fill, circle, black, inner sep=1pt] {};};
\draw (0,0)--(1,1) (0,1)--(1,0);
\node at (0.5,1.5) {$I(C_{2,2})$};
\end{tikzpicture}& \hspace{0.2in} &
\begin{tikzpicture}
[scale=0.6]
\node at (0,0) [fill, circle, black, inner sep=1pt] {};
\node at (0,1) [fill, circle, black, inner sep=1pt] {};
\node at (0,1.5) {$I(M_{2,1})$};
\end{tikzpicture}& \hspace{0.2in} &
\begin{tikzpicture}
[scale=0.6]
\foreach \x in {(0,0),(1,0),(0,1),(1,1)} {\node at \x [fill, circle, black, inner sep=1pt] {};};
\node at (0.5,1.5) {$I(M_{2,2})$};
\end{tikzpicture}
\end{tabular}
\caption{$C_{2,1}$, $C_{2,2}$, $M_{2,1}$, $M_{2,2}$ and their independence complexes}
\label{mainm2basecases}
\end{figure}

In order to determine $I(C_{3,n})$ and $I(M_{3,n})$, it is sufficient to show that 
\begin{align*}
I(C_{3,1}) = {\dia}^{-1}, & &I(C_{3,2}) \searrow {\dia}^1, & &I(C_{3,3}) \searrow {\dia}^1, & &I(C_{3,4}) \simpsimeq {\bigvee}_3 {\dia}^2 \\
I(M_{3,1}) = {\dia}^0 , & &I(M_{3,2}) \simpsimeq {\dia}^0 , & &I(M_{3,3}) \searrow {\dia}^2 , & &I(M_{3,4}) \simpsimeq {\bigvee}_5 {\dia}^2.
\end{align*}

\begin{itemize}
\item It is clear that $I(C_{3,1}) = {\dia}^{-1}$ and $I(C_{3,2}) \searrow {\dia}^1$ (see Figure \ref{C312M312}). Figure \ref{C312M312} also indicates that $I(C_{3,2}) \searrow {\dia}^1$ by
$$\del(\overline{1}, 2), \del(\overline{2},1)$$
and that $I(M_{3,2}) \searrow {\dia}^0$ by
$$\del(\widehat{1},1), \del(\widehat{2},2), \del(\overline{1},2), \del(\overline{2},1).$$
\begin{figure}[H]
\centering
\begin{tabular}{cccccccccccc}
\begin{tikzpicture}
[scale=0.6]
\foreach \y in {1,2,3} {\node at (0,\y) [fill,circle,inner sep=1pt] {};};
\draw (0,1)--(0,3);
\draw (0,1) to [out=-30, in=-90] (1,1);
\draw (1,1) to [out=90, in=30] (0,1);
\draw (0,2) to [out=-30, in=-90] (1,2);
\draw (1,2) to [out=90, in=30] (0,2);
\draw (0,3) to [out=-30, in=-90] (1,3);
\draw (1,3) to [out=90, in=30] (0,3);
\node at (0.5, 4) {$C_{3,1}$};
\end{tikzpicture}
& &
\begin{tikzpicture}
[scale=0.6]
\foreach \y in {1,2,3} {\node at (0,\y) [fill,circle,inner sep=1pt] {};};
\draw (0,1)--(0,3);
\draw (0,3) to [out=250, in=110] (0,1);
\draw (0,2) to [out=-30, in=-90] (1,2);
\draw (1,2) to [out=90, in=30] (0,2);
\draw (0,1) to [out=140, in=220] (0,3);
\node at (0.5,4) {$M_{3,1}$};
\end{tikzpicture}
& &
\begin{tikzpicture}
[scale=0.6]
\foreach \x in {0,1} {\foreach \y in {1,2,3} {\node at (\x,\y) [fill,circle,inner sep=1pt] {};};};
\draw (0,1) grid (1,3);
\draw (0,1) to [out=40, in=140] (1,1);
\draw (0,2) to [out=40, in=140] (1,2);
\draw (0,3) to [out=40, in=140] (1,3);
\node at (0.5, 4) {$C_{3,2}$};
\end{tikzpicture}
&\raisebox{7mm}{$\rightarrow$}&
\begin{tikzpicture}
[scale=0.6]
\foreach \x in {0,1} {\foreach \y in {1,3} {\node at (\x,\y) [fill,circle,inner sep=1pt] {};};};
\draw (0,1)--(1,1) (0,3)--(1,3);
\draw (0,1) to [out=40, in=140] (1,1);
\draw [dashed] (0,2) to [out=40, in=140] (1,2);
\draw [dashed] (0,2)--(1,2) (0,1)--(0,3) (1,1)--(1,3);
\draw (0,3) to [out=40, in=140] (1,3);
\node at (0,2) [circle, draw=black, fill=white, inner sep=1pt] {};
\node at (1,2) [circle, draw=black, fill=white, inner sep=1pt] {};
\node at (0.5, 4) {};
\end{tikzpicture}
& &
\begin{tikzpicture}
[scale=0.6]
\foreach \x in {0,1} {\foreach \y in {1,2,3} {\node at (\x,\y) [fill,circle,inner sep=1pt] {};};};
\draw (0,1) grid (1,3);
\draw (0,1) to [out=80, in=190] (1,3);
\draw (0,2) to [out=40, in=140] (1,2);
\draw (0,3) to [out=350, in=100] (1,1);
\node at (0.5, 4) {$M_{3,2}$};
\end{tikzpicture}
&\raisebox{7mm}{$\rightarrow$}&
\begin{tikzpicture}
[scale=0.6]
\foreach \x in {0,1} {\foreach \y in {1} {\node at (\x,\y) [fill,circle,inner sep=1pt] {};};};
\draw (0,1)--(1,1);
\draw [dashed] (0,1) to [out=80, in=190] (1,3);
\draw [dashed] (0,2) to [out=40, in=140] (1,2);
\draw [dashed] (0,3) to [out=350, in=100] (1,1);
\draw [dashed] (0,2)--(0,3)--(1,3)--(1,2) (0,1)--(0,2)--(1,2)--(1,1);
\node at (0.5, 4) {};
\node at (0,3) [circle, draw=black, fill=white, inner sep=1pt] {};
\node at (1,3) [circle, draw=black, fill=white, inner sep=1pt] {};
\node at (0,2) [circle, draw=black, fill=white, inner sep=1pt] {};
\node at (1,2) [circle, draw=black, fill=white, inner sep=1pt] {};
\end{tikzpicture}
\end{tabular}
\caption{$C_{3,1}$, $M_{3,1}$, $C_{3,2}$, $M_{3,2}$ and the operations}
\label{C312M312}
\end{figure}
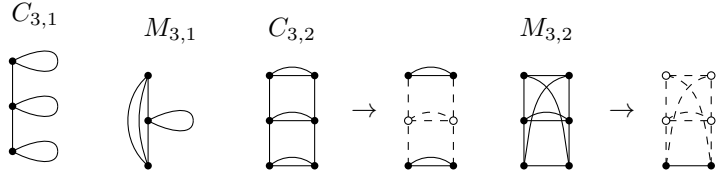

\item $I(C_{3,3}) \searrow {\dia}^1$ is obtained by
\begin{align*}
&\add(1\widehat{2},\overline{3}), \add(1\widehat{3},\overline{2}), \del(1,\widehat{1}), \del(\overline{2},3),
\del(\overline{3},2), \del(\widehat{3},\overline{1}), \del(\overline{1},\widehat{2})
\end{align*}
(see Figure \ref{C33}).
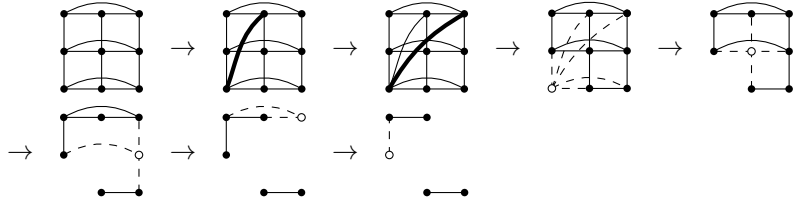
\begin{figure}[H]
\centering
\begin{tabular}{cccccccccc}
&
{\begin{tikzpicture}[scale=0.5]
\foreach \x in {(1,1),(1,2),(1,3),(2,1),(2,2),(2,3),(3,1),(3,2),(3,3)} {\node at \x [fill,circle,inner sep=1pt] {};};
\draw (1,1)--(1,2)--(1,3) (2,1)--(2,2)--(2,3) (3,1)--(3,2)--(3,3) (1,1)--(2,1)--(3,1) (1,2)--(2,2)--(3,2) (1,3)--(2,3)--(3,3);
\draw (1,1) to [out=30, in=150] (3,1);
\draw (1,2) to [out=30, in=150] (3,2);
\draw (1,3) to [out=30, in=150] (3,3);
\end{tikzpicture}}
&\raisebox{5mm}{$\rightarrow$}&
{\begin{tikzpicture}[scale=0.5]
\foreach \x in {(1,1),(1,2),(1,3),(2,1),(2,2),(2,3),(3,1),(3,2),(3,3)} {\node at \x [fill,circle,inner sep=1pt] {};};
\draw (1,1)--(1,2)--(1,3) (2,1)--(2,2)--(2,3) (3,1)--(3,2)--(3,3) (1,1)--(2,1)--(3,1) (1,2)--(2,2)--(3,2) (1,3)--(2,3)--(3,3);
\draw (1,1) to [out=30, in=150] (3,1);
\draw (1,2) to [out=30, in=150] (3,2);
\draw (1,3) to [out=30, in=150] (3,3);
\draw [ultra thick] (1,1) to [out=70, in=220] (2,3);
\end{tikzpicture}}
&\raisebox{5mm}{$\rightarrow$}&
{\begin{tikzpicture}[scale=0.5]
\foreach \x in {(1,1),(1,2),(1,3),(2,1),(2,2),(2,3),(3,1),(3,2),(3,3)} {\node at \x [fill,circle,inner sep=1pt] {};};
\draw (1,1)--(1,2)--(1,3) (2,1)--(2,2)--(2,3) (3,1)--(3,2)--(3,3) (1,1)--(2,1)--(3,1) (1,2)--(2,2)--(3,2) (1,3)--(2,3)--(3,3);
\draw (1,1) to [out=30, in=150] (3,1);
\draw (1,2) to [out=30, in=150] (3,2);
\draw (1,3) to [out=30, in=150] (3,3);
\draw (1,1) to [out=70, in=220] (2,3);
\draw [ultra thick] (1,1) to [out=60, in=210] (3,3);
\end{tikzpicture}}
&\raisebox{5mm}{$\rightarrow$}&
{\begin{tikzpicture}[scale=0.5]
\foreach \x in {(1,2),(1,3),(2,1),(2,2),(2,3),(3,1),(3,2),(3,3)} {\node at \x [fill,circle,inner sep=1pt] {};};
\node at (3,3) [fill,circle,inner sep=1pt] {};
\draw (1,2)--(1,3) (2,1)--(2,2)--(2,3) (3,1)--(3,2)--(3,3) (2,1)--(3,1) (1,2)--(2,2)--(3,2) (1,3)--(2,3)--(3,3);
\draw (1,2) to [out=30, in=150] (3,2);
\draw (1,3) to [out=30, in=150] (3,3);
\draw [dashed] (1,1) to [out=30, in=150] (3,1);
\draw [dashed] (1,1) to [out=70, in=220] (2,3);
\draw [dashed] (1,1) to [out=60, in=210] (3,3);
\draw [dashed] (1,2)--(1,1)--(2,1);
\node at (1,1) [circle, draw=black, fill=white, inner sep=1pt] {};
\end{tikzpicture}}
&\raisebox{5mm}{$\rightarrow$}&
{\begin{tikzpicture}[scale=0.5]
\foreach \x in {(1,2),(1,3),(2,1),(2,3),(3,1),(3,2),(3,3)} {\node at \x [fill,circle,inner sep=1pt] {};};
\draw (1,2)--(1,3) (3,1)--(3,2)--(3,3) (2,1)--(3,1) (1,3)--(2,3)--(3,3);
\draw (1,2) to [out=30, in=150] (3,2);
\draw (1,3) to [out=30, in=150] (3,3);
\draw [dashed] (1,2)--(3,2) (2,1)--(2,3);
\node at (2,2) [circle, draw=black, fill=white, inner sep=1pt] {};
\end{tikzpicture}}
\\
\raisebox{5mm}{$\rightarrow$}&
{\begin{tikzpicture}[scale=0.5]
\foreach \x in {(1,2),(1,3),(2,1),(2,3),(3,1),(3,3)} {\node at \x [fill,circle,inner sep=1pt] {};};
\draw (1,2)--(1,3) (2,1)--(3,1) (1,3)--(2,3)--(3,3);
\draw (1,3) to [out=30, in=150] (3,3);
\draw [dashed] (3,1)--(3,3);
\draw [dashed] (1,2) to [out=30, in=150] (3,2);
\node at (3,2) [circle, draw=black, fill=white, inner sep=1pt] {};
\end{tikzpicture}}
&\raisebox{5mm}{$\rightarrow$}&
{\begin{tikzpicture}[scale=0.5]
\foreach \x in {(1,2),(1,3),(2,1),(2,3),(3,1)} {\node at \x [fill,circle,inner sep=1pt] {};};
\draw (1,2)--(1,3) (2,1)--(3,1) (1,3)--(2,3);
\draw [dashed] (2,3)--(3,3);
\draw [dashed] (1,3) to [out=30, in=150] (3,3);
\node at (3,3) [circle, draw=black, fill=white, inner sep=1pt] {};
\end{tikzpicture}}
&\raisebox{5mm}{$\rightarrow$}&
{\begin{tikzpicture}[scale=0.5]
\foreach \x in {(1,3),(2,1),(2,3),(3,1)} {\node at \x [fill,circle,inner sep=1pt] {};};
\draw (2,1)--(3,1) (1,3)--(2,3);
\draw [dashed] (1,2)--(1,3);
\node at (1,2) [circle, draw=black, fill=white, inner sep=1pt] {};
\end{tikzpicture}}
\end{tabular}
\caption{A  sequence of operations which induces $I(C_{3,3}) \searrow I(e \sqcup e)$}
\label{C33}
\end{figure}

\item We show $I(C_{3,4}) \simpsimeq I(C_{2,4} \sqcup e) \simpsimeq \bigvee_3 {\dia}^2$. The first transformation is performed as follows:
\begin{align*}
&\add(1\widehat{3},\overline{2}), \del(1\overline{1},\widehat{2}), \del(3\overline{3},1), \add(\overline{2}\overline{4},3), \\
&\del(2\overline{2},4), \del(4,2), \del(1,3), \del(\overline{2}\overline{4},\widehat{3})
\end{align*}
(see Figure \ref{C34}). The second transformation is obtained by
$$I(C_{2,4}) \simpsimeq \Sigma I(M_{2,2}) \simpsimeq {\bigvee}_3 {\dia}^1 .$$
\begin{figure}[H]
\centering
\begin{tabular}{cccccccc}
&
{\begin{tikzpicture}[scale=0.5]
\foreach \x in {1,2,3,4} {\foreach \y in {1,2,3} {\node at (\x,\y) [fill,circle,inner sep=1pt] {};};};
\draw (1,1)--(1,2)--(1,3) (2,1)--(2,2)--(2,3) (3,1)--(3,2)--(3,3) (4,1)--(4,2)--(4,3); 
\draw (1,1)--(2,1)--(3,1)--(4,1) (1,2)--(2,2)--(3,2)--(4,2) (1,3)--(2,3)--(3,3)--(4,3);
\draw (1,1) to [out=30, in=150] (4,1);
\draw (1,2) to [out=30, in=150] (4,2);
\draw (1,3) to [out=30, in=150] (4,3);
\end{tikzpicture}}
&\raisebox{5mm}{$\rightarrow$}&
{\begin{tikzpicture}[scale=0.5]
\foreach \x in {1,2,3,4} {\foreach \y in {1,2,3} {\node at (\x,\y) [fill,circle,inner sep=1pt] {};};};
\draw (1,1)--(1,2)--(1,3) (2,1)--(2,2)--(2,3) (3,1)--(3,2)--(3,3) (4,1)--(4,2)--(4,3); 
\draw (1,1)--(2,1)--(3,1)--(4,1) (1,2)--(2,2)--(3,2)--(4,2) (1,3)--(2,3)--(3,3)--(4,3);
\draw (1,1) to [out=30, in=150] (4,1);
\draw (1,2) to [out=30, in=150] (4,2);
\draw (1,3) to [out=30, in=150] (4,3);
\draw [ultra thick] (1,1) to [out=60, in=210] (3,3);
\end{tikzpicture}}
&\raisebox{5mm}{$\rightarrow$}&
{\begin{tikzpicture}[scale=0.5]
\foreach \x in {1,2,3,4} {\foreach \y in {1,2,3} {\node at (\x,\y) [fill,circle,inner sep=1pt] {};};};
\draw (1,2)--(1,3) (2,1)--(2,2)--(2,3) (3,1)--(3,2)--(3,3) (4,1)--(4,2)--(4,3); 
\draw (1,1)--(2,1)--(3,1)--(4,1) (1,2)--(2,2)--(3,2)--(4,2) (1,3)--(2,3)--(3,3)--(4,3);
\draw (1,1) to [out=30, in=150] (4,1);
\draw (1,2) to [out=30, in=150] (4,2);
\draw (1,3) to [out=30, in=150] (4,3);
\draw (1,1) to [out=60, in=210] (3,3);
\draw [dashed] (1,1)--(1,2);
\end{tikzpicture}}
&\raisebox{5mm}{$\rightarrow$}&
{\begin{tikzpicture}[scale=0.5]
\foreach \x in {1,2,3,4} {\foreach \y in {1,2,3} {\node at (\x,\y) [fill,circle,inner sep=1pt] {};};};
\draw (1,2)--(1,3) (2,1)--(2,2)--(2,3) (3,2)--(3,3) (4,1)--(4,2)--(4,3); 
\draw (1,1)--(2,1)--(3,1)--(4,1) (1,2)--(2,2)--(3,2)--(4,2) (1,3)--(2,3)--(3,3)--(4,3);
\draw (1,1) to [out=30, in=150] (4,1);
\draw (1,2) to [out=30, in=150] (4,2);
\draw (1,3) to [out=30, in=150] (4,3);
\draw (1,1) to [out=60, in=210] (3,3);
\draw [dashed] (3,1)--(3,2);
\end{tikzpicture}}
\\
\raisebox{5mm}{$\rightarrow$}&
{\begin{tikzpicture}[scale=0.5]
\foreach \x in {1,2,3,4} {\foreach \y in {1,2,3} {\node at (\x,\y) [fill,circle,inner sep=1pt] {};};};
\draw (1,2)--(1,3) (2,1)--(2,2)--(2,3) (3,2)--(3,3) (4,1)--(4,2)--(4,3); 
\draw (1,1)--(2,1)--(3,1)--(4,1) (1,2)--(2,2)--(3,2)--(4,2) (1,3)--(2,3)--(3,3)--(4,3);
\draw (1,1) to [out=30, in=150] (4,1);
\draw (1,2) to [out=30, in=150] (4,2);
\draw (1,3) to [out=30, in=150] (4,3);
\draw (1,1) to [out=60, in=210] (3,3);
\draw [ultra thick] (2,2) to [out=20, in=160] (4,2);
\end{tikzpicture}}
&\raisebox{5mm}{$\rightarrow$}&
{\begin{tikzpicture}[scale=0.5]
\foreach \x in {1,2,3,4} {\foreach \y in {1,2,3} {\node at (\x,\y) [fill,circle,inner sep=1pt] {};};};
\draw (1,2)--(1,3) (2,2)--(2,3) (3,2)--(3,3) (4,1)--(4,2)--(4,3); 
\draw (1,1)--(2,1)--(3,1)--(4,1) (1,2)--(2,2)--(3,2)--(4,2) (1,3)--(2,3)--(3,3)--(4,3);
\draw (1,1) to [out=30, in=150] (4,1);
\draw (1,2) to [out=30, in=150] (4,2);
\draw (1,3) to [out=30, in=150] (4,3);
\draw (1,1) to [out=60, in=210] (3,3);
\draw (2,2) to [out=20, in=160] (4,2);
\draw [dashed] (2,1)--(2,2);
\end{tikzpicture}}
&\raisebox{5mm}{$\rightarrow$}&
{\begin{tikzpicture}[scale=0.5]
\node at (1,1) [fill,circle,inner sep=1pt] {};
\node at (2,1) [fill,circle,inner sep=1pt] {};
\node at (3,1) [fill,circle,inner sep=1pt] {};
\foreach \x in {1,2,3,4} {\foreach \y in {2,3} {\node at (\x,\y) [fill,circle,inner sep=1pt] {};};};
\draw (1,2)--(1,3) (2,2)--(2,3) (3,2)--(3,3) (4,2)--(4,3); 
\draw (1,1)--(2,1)--(3,1) (1,2)--(2,2)--(3,2)--(4,2) (1,3)--(2,3)--(3,3)--(4,3);
\draw (1,2) to [out=30, in=150] (4,2);
\draw (1,3) to [out=30, in=150] (4,3);
\draw (1,1) to [out=60, in=210] (3,3);
\draw (2,2) to [out=20, in=160] (4,2);
\draw [dashed] (3,1)--(4,1)--(4,2);
\draw [dashed] (1,1) to [out=30, in=150] (4,1);
\node at (4,1) [circle, draw=black, fill=white, inner sep=1pt] {};
\end{tikzpicture}}
&\raisebox{5mm}{$\rightarrow$}&
{\begin{tikzpicture}[scale=0.5]
\node at (2,1) [fill,circle,inner sep=1pt] {};
\node at (3,1) [fill,circle,inner sep=1pt] {};
\foreach \x in {1,2,3,4} {\foreach \y in {2,3} {\node at (\x,\y) [fill,circle,inner sep=1pt] {};};};
\draw (1,2)--(1,3) (2,2)--(2,3) (3,2)--(3,3) (4,2)--(4,3); 
\draw (2,1)--(3,1) (1,2)--(2,2)--(3,2)--(4,2) (1,3)--(2,3)--(3,3)--(4,3);
\draw (1,2) to [out=30, in=150] (4,2);
\draw (1,3) to [out=30, in=150] (4,3);
\draw (2,2) to [out=20, in=160] (4,2);
\draw [dashed] (1,1) to [out=60, in=210] (3,3);
\draw [dashed] (1,1)--(2,1);
\node at (1,1) [circle, draw=black, fill=white, inner sep=1pt] {};
\end{tikzpicture}}
\\
\raisebox{5mm}{$\rightarrow$}&
{\begin{tikzpicture}[scale=0.5]
\node at (2,1) [fill,circle,inner sep=1pt] {};
\node at (3,1) [fill,circle,inner sep=1pt] {};
\foreach \x in {1,2,3,4} {\foreach \y in {2,3} {\node at (\x,\y) [fill,circle,inner sep=1pt] {};};};
\draw (1,2)--(1,3) (2,2)--(2,3) (3,2)--(3,3) (4,2)--(4,3); 
\draw (2,1)--(3,1) (1,2)--(2,2)--(3,2)--(4,2) (1,3)--(2,3)--(3,3)--(4,3);
\draw (1,2) to [out=30, in=150] (4,2);
\draw (1,3) to [out=30, in=150] (4,3);
\draw [dashed] (2,2) to [out=20, in=160] (4,2);
\end{tikzpicture}}
\end{tabular}
\caption{A sequence of operations which induces $I(C_{3,4}) \simpsimeq I(C_{2,4} \sqcup e)$}
\label{C34}
\end{figure}

\item We show $I(M_{3,3}) \searrow I(C_{1,8}) \searrow {\dia}^2$. The first transformation is performed as follows:
\begin{align*}
&\del(1\widehat{3},\overline{2}), \del(\widehat{1}3,\overline{2}), \del(\overline{1}\overline{3},2), \del(\overline{2},1)
\end{align*}
(see Figure \ref{M33}). The second transformation is obtained by Corollary \ref{Cnsimple}.
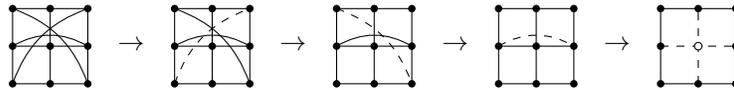
\begin{figure}[H]
\centering
\begin{tabular}{cccccccccc}
&
{\begin{tikzpicture}[scale=0.5]
\foreach \x in {(1,1),(1,2),(1,3),(2,1),(2,2),(2,3),(3,1),(3,2),(3,3)} {\node at \x [fill,circle,inner sep=1pt] {};};
\draw (1,1)--(1,2)--(1,3) (2,1)--(2,2)--(2,3) (3,1)--(3,2)--(3,3) (1,1)--(2,1)--(3,1) (1,2)--(2,2)--(3,2) (1,3)--(2,3)--(3,3);
\draw (1,1) to [out=70, in=200] (3,3);
\draw (1,2) to [out=30, in=150] (3,2);
\draw (1,3) to [out=340, in=110] (3,1);
\end{tikzpicture}}
&\raisebox{5mm}{$\rightarrow$}&
{\begin{tikzpicture}[scale=0.5]
\foreach \x in {(1,1),(1,2),(1,3),(2,1),(2,2),(2,3),(3,1),(3,2),(3,3)} {\node at \x [fill,circle,inner sep=1pt] {};};
\draw (1,1)--(1,2)--(1,3) (2,1)--(2,2)--(2,3) (3,1)--(3,2)--(3,3) (1,1)--(2,1)--(3,1) (1,2)--(2,2)--(3,2) (1,3)--(2,3)--(3,3);
\draw (1,2) to [out=30, in=150] (3,2);
\draw (1,3) to [out=340, in=110] (3,1);
\draw [dashed] (1,1) to [out=70, in=200] (3,3);
\end{tikzpicture}}
&\raisebox{5mm}{$\rightarrow$}&
{\begin{tikzpicture}[scale=0.5]
\foreach \x in {(1,1),(1,2),(1,3),(2,1),(2,2),(2,3),(3,1),(3,2),(3,3)} {\node at \x [fill,circle,inner sep=1pt] {};};
\draw (1,1)--(1,2)--(1,3) (2,1)--(2,2)--(2,3) (3,1)--(3,2)--(3,3) (1,1)--(2,1)--(3,1) (1,2)--(2,2)--(3,2) (1,3)--(2,3)--(3,3);
\draw (1,2) to [out=30, in=150] (3,2);
\draw [dashed] (1,3) to [out=340, in=110] (3,1);
\end{tikzpicture}}
&\raisebox{5mm}{$\rightarrow$}&
{\begin{tikzpicture}[scale=0.5]
\foreach \x in {(1,1),(1,2),(1,3),(2,1),(2,2),(2,3),(3,1),(3,2),(3,3)} {\node at \x [fill,circle,inner sep=1pt] {};};
\draw (1,1)--(1,2)--(1,3) (2,1)--(2,2)--(2,3) (3,1)--(3,2)--(3,3) (1,1)--(2,1)--(3,1) (1,2)--(2,2)--(3,2) (1,3)--(2,3)--(3,3);
\draw [dashed] (1,2) to [out=30, in=150] (3,2);
\end{tikzpicture}}
&\raisebox{5mm}{$\rightarrow$}&
{\begin{tikzpicture}[scale=0.5]
\foreach \x in {(1,1),(1,2),(1,3),(2,1),(2,3),(3,1),(3,2),(3,3)} {\node at \x [fill,circle,inner sep=1pt] {};};
\draw (1,1)--(1,2)--(1,3) (3,1)--(3,2)--(3,3) (1,1)--(2,1)--(3,1) (1,3)--(2,3)--(3,3);
\draw [dashed] (1,2)--(3,2) (2,1)--(2,3);
\node at (2,2) [circle, draw=black, fill=white, inner sep=1pt] {};
\end{tikzpicture}}
\end{tabular}
\caption{A sequence of operations which induces $I(M_{3,3}) \simpsimeq I(C_{1,8})$}
\label{M33}
\end{figure}

\item We show that $I(M_{3,4}) \simpsimeq \bigvee_5 {\dia}^2$. Let $G$ be a graph in Figure \ref{M34figure1}, where $I(G) \simpsimeq \bigvee_5 {\dia}^1$.
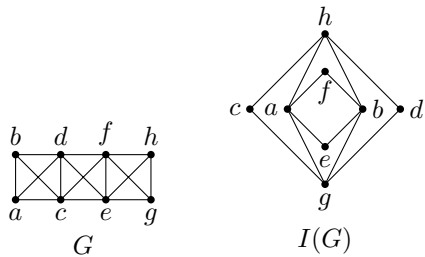
\begin{figure}[H]
\centering
\begin{tabular}{ccc}
\begin{tikzpicture}
[scale=0.6]
\foreach \x in {1,2,3,4} {\foreach \y in {1,2} {\node at (\x,\y) [fill,circle,inner sep=1pt] {};};};
\draw (1,1) grid (4,2);
\draw (1,1)--(2,2)--(3,1)--(4,2) (1,2)--(2,1)--(3,2)--(4,1);
\node at (1,1) [below] {$a$};
\node at (1,2) [above] {$b$};
\node at (2,1) [below] {$c$};
\node at (2,2) [above] {$d$};
\node at (3,1) [below] {$e$};
\node at (3,2) [above] {$f$};
\node at (4,1) [below] {$g$};
\node at (4,2) [above] {$h$};
\node at (2.5, 0) {$G$};
\end{tikzpicture}
& &
\begin{tikzpicture}
[scale=0.5]
\foreach \x in {(-2,0),(-1,0),(1,0),(2,0),(0,-2),(0,-1),(0,1),(0,2)} {\node at \x [fill,circle,inner sep=1pt] {};};
\node at (-2,0) [left] {$c$};
\node at (-1,0) [left] {$a$};
\node at (1,0) [right] {$b$};
\node at (2,0) [right] {$d$};
\node at (0,-2) [below] {$g$};
\node at (0,-1) [below] {$e$};
\node at (0,1) [below] {$f$};
\node at (0,2) [above] {$h$};
\draw (-2,0)--(0,-2)--(2,0)--(0,2)--cycle;
\draw (-1,0)--(0,-2)--(1,0)--(0,2)--cycle;
\draw (-1,0)--(0,-1)--(1,0)--(0,1)--cycle;
\node at (0, -3.5) {$I(G)$};
\end{tikzpicture}
\end{tabular}
\caption{The graph $G$ and its independence complex}
\label{M34figure1}
\end{figure}

Then, we obtain $I(M_{3,4}) \simpsimeq I(G \sqcup e)$ by
\begin{align*}
&\add(\overline{1}\overline{3},2), \add(\overline{2}\overline{4},3), \add(1\overline{2},\widehat{3}), \add(\overline{1}2,4), 
\add(3\overline{4},1), \\
&\add(\overline{3}4,\widehat{2}), \del(\overline{1}\widehat{1},3), \del(\overline{2}\widehat{2},\widehat{4}), 
\del(\overline{3}\widehat{3},\widehat{1}), \add(4\widehat{4},\widehat{2}), \\
&\del(\widehat{1}4,\widehat{3}), \del(\widehat{3},\widehat{1}), \del(23,\widehat{4}), \add(\overline{3}\widehat{4},2), \del(\widehat{4},3)
\end{align*}
(see Figure \ref{M34figure2}).
\begin{figure}[H]
\centering
\begin{tabular}{cccccccc}
&
{\begin{tikzpicture}[scale=0.5]
\foreach \x in {1,2,3,4} {\foreach \y in {1,2,3} {\node at (\x,\y) [fill,circle,inner sep=1pt] {};};};
\draw (1,1)--(1,2)--(1,3) (2,1)--(2,2)--(2,3) (3,1)--(3,2)--(3,3) (4,1)--(4,2)--(4,3); 
\draw (1,1)--(2,1)--(3,1)--(4,1) (1,2)--(2,2)--(3,2)--(4,2) (1,3)--(2,3)--(3,3)--(4,3);
\draw (1,1) to [out=70, in=190] (4,3);
\draw (1,2) to [out=30, in=150] (4,2);
\draw (1,3) to [out=350, in=110] (4,1);
\end{tikzpicture}}
&\raisebox{5mm}{$\rightarrow$}&
{\begin{tikzpicture}[scale=0.5]
\foreach \x in {1,2,3,4} {\foreach \y in {1,2,3} {\node at (\x,\y) [fill,circle,inner sep=1pt] {};};};
\draw (1,1)--(1,2)--(1,3) (2,1)--(2,2)--(2,3) (3,1)--(3,2)--(3,3) (4,1)--(4,2)--(4,3); 
\draw (1,1)--(2,1)--(3,1)--(4,1) (1,2)--(2,2)--(3,2)--(4,2) (1,3)--(2,3)--(3,3)--(4,3);
\draw (1,1) to [out=70, in=190] (4,3);
\draw (1,2) to [out=30, in=150] (4,2);
\draw (1,3) to [out=350, in=110] (4,1);
\draw [ultra thick] (1,2) to [out=20, in=160] (3,2);
\end{tikzpicture}}
&\raisebox{5mm}{$\rightarrow$}&
{\begin{tikzpicture}[scale=0.5]
\foreach \x in {1,2,3,4} {\foreach \y in {1,2,3} {\node at (\x,\y) [fill,circle,inner sep=1pt] {};};};
\draw (1,1)--(1,2)--(1,3) (2,1)--(2,2)--(2,3) (3,1)--(3,2)--(3,3) (4,1)--(4,2)--(4,3); 
\draw (1,1)--(2,1)--(3,1)--(4,1) (1,2)--(2,2)--(3,2)--(4,2) (1,3)--(2,3)--(3,3)--(4,3);
\draw (1,1) to [out=70, in=190] (4,3);
\draw (1,2) to [out=30, in=150] (4,2);
\draw (1,3) to [out=350, in=110] (4,1);
\draw (1,2) to [out=20, in=160] (3,2);
\draw [ultra thick] (2,2) to [out=20, in=160] (4,2);
\end{tikzpicture}}
&\raisebox{5mm}{$\rightarrow$}&
{\begin{tikzpicture}[scale=0.5]
\foreach \x in {1,2,3,4} {\foreach \y in {1,2,3} {\node at (\x,\y) [fill,circle,inner sep=1pt] {};};};
\draw (1,1)--(1,2)--(1,3) (2,1)--(2,2)--(2,3) (3,1)--(3,2)--(3,3) (4,1)--(4,2)--(4,3); 
\draw (1,1)--(2,1)--(3,1)--(4,1) (1,2)--(2,2)--(3,2)--(4,2) (1,3)--(2,3)--(3,3)--(4,3);
\draw (1,1) to [out=70, in=190] (4,3);
\draw (1,2) to [out=30, in=150] (4,2);
\draw (1,3) to [out=350, in=110] (4,1);
\draw (1,2) to [out=20, in=160] (3,2);
\draw (2,2) to [out=20, in=160] (4,2);
\draw [ultra thick] (1,1)--(2,2);
\end{tikzpicture}}
\\
\raisebox{5mm}{$\rightarrow$}&
{\begin{tikzpicture}[scale=0.5]
\foreach \x in {1,2,3,4} {\foreach \y in {1,2,3} {\node at (\x,\y) [fill,circle,inner sep=1pt] {};};};
\draw (1,1)--(1,2)--(1,3) (2,1)--(2,2)--(2,3) (3,1)--(3,2)--(3,3) (4,1)--(4,2)--(4,3); 
\draw (1,1)--(2,1)--(3,1)--(4,1) (1,2)--(2,2)--(3,2)--(4,2) (1,3)--(2,3)--(3,3)--(4,3);
\draw (1,1) to [out=70, in=190] (4,3);
\draw (1,2) to [out=30, in=150] (4,2);
\draw (1,3) to [out=350, in=110] (4,1);
\draw (1,2) to [out=20, in=160] (3,2);
\draw (2,2) to [out=20, in=160] (4,2);
\draw (1,1)--(2,2);
\draw [ultra thick] (1,2)--(2,1);
\end{tikzpicture}}
&\raisebox{5mm}{$\rightarrow$}&
{\begin{tikzpicture}[scale=0.5]
\foreach \x in {1,2,3,4} {\foreach \y in {1,2,3} {\node at (\x,\y) [fill,circle,inner sep=1pt] {};};};
\draw (1,1)--(1,2)--(1,3) (2,1)--(2,2)--(2,3) (3,1)--(3,2)--(3,3) (4,1)--(4,2)--(4,3); 
\draw (1,1)--(2,1)--(3,1)--(4,1) (1,2)--(2,2)--(3,2)--(4,2) (1,3)--(2,3)--(3,3)--(4,3);
\draw (1,1) to [out=70, in=190] (4,3);
\draw (1,2) to [out=30, in=150] (4,2);
\draw (1,3) to [out=350, in=110] (4,1);
\draw (1,2) to [out=20, in=160] (3,2);
\draw (2,2) to [out=20, in=160] (4,2);
\draw (1,1)--(2,2) (1,2)--(2,1);
\draw [ultra thick] (3,1)--(4,2);
\end{tikzpicture}}
&\raisebox{5mm}{$\rightarrow$}&
{\begin{tikzpicture}[scale=0.5]
\foreach \x in {1,2,3,4} {\foreach \y in {1,2,3} {\node at (\x,\y) [fill,circle,inner sep=1pt] {};};};
\draw (1,1)--(1,2)--(1,3) (2,1)--(2,2)--(2,3) (3,1)--(3,2)--(3,3) (4,1)--(4,2)--(4,3); 
\draw (1,1)--(2,1)--(3,1)--(4,1) (1,2)--(2,2)--(3,2)--(4,2) (1,3)--(2,3)--(3,3)--(4,3);
\draw (1,1) to [out=70, in=190] (4,3);
\draw (1,2) to [out=30, in=150] (4,2);
\draw (1,3) to [out=350, in=110] (4,1);
\draw (1,2) to [out=20, in=160] (3,2);
\draw (2,2) to [out=20, in=160] (4,2);
\draw (1,1)--(2,2) (1,2)--(2,1) (3,1)--(4,2);
\draw [ultra thick] (3,2)--(4,1);
\end{tikzpicture}}
&\raisebox{5mm}{$\rightarrow$}&
{\begin{tikzpicture}[scale=0.5]
\foreach \x in {1,2,3,4} {\foreach \y in {1,2,3} {\node at (\x,\y) [fill,circle,inner sep=1pt] {};};};
\draw (1,1)--(1,2) (2,1)--(2,2)--(2,3) (3,1)--(3,2)--(3,3) (4,1)--(4,2)--(4,3); 
\draw (1,1)--(2,1)--(3,1)--(4,1) (1,2)--(2,2)--(3,2)--(4,2) (1,3)--(2,3)--(3,3)--(4,3);
\draw (1,1) to [out=70, in=190] (4,3);
\draw (1,2) to [out=30, in=150] (4,2);
\draw (1,3) to [out=350, in=110] (4,1);
\draw (1,2) to [out=20, in=160] (3,2);
\draw (2,2) to [out=20, in=160] (4,2);
\draw (1,1)--(2,2) (1,2)--(2,1) (3,1)--(4,2) (3,2)--(4,1);
\draw [dashed] (1,2)--(1,3);
\end{tikzpicture}}
\\
\raisebox{5mm}{$\rightarrow$}&
{\begin{tikzpicture}[scale=0.5]
\foreach \x in {1,2,3,4} {\foreach \y in {1,2,3} {\node at (\x,\y) [fill,circle,inner sep=1pt] {};};};
\draw (1,1)--(1,2) (2,1)--(2,2) (3,1)--(3,2)--(3,3) (4,1)--(4,2)--(4,3); 
\draw (1,1)--(2,1)--(3,1)--(4,1) (1,2)--(2,2)--(3,2)--(4,2) (1,3)--(2,3)--(3,3)--(4,3);
\draw (1,1) to [out=70, in=190] (4,3);
\draw (1,2) to [out=30, in=150] (4,2);
\draw (1,3) to [out=350, in=110] (4,1);
\draw (1,2) to [out=20, in=160] (3,2);
\draw (2,2) to [out=20, in=160] (4,2);
\draw (1,1)--(2,2) (1,2)--(2,1) (3,1)--(4,2) (3,2)--(4,1);
\draw [dashed] (2,2)--(2,3);
\end{tikzpicture}}
&\raisebox{5mm}{$\rightarrow$}&
{\begin{tikzpicture}[scale=0.5]
\foreach \x in {1,2,3,4} {\foreach \y in {1,2,3} {\node at (\x,\y) [fill,circle,inner sep=1pt] {};};};
\draw (1,1)--(1,2) (2,1)--(2,2) (3,1)--(3,2) (4,1)--(4,2)--(4,3); 
\draw (1,1)--(2,1)--(3,1)--(4,1) (1,2)--(2,2)--(3,2)--(4,2) (1,3)--(2,3)--(3,3)--(4,3);
\draw (1,1) to [out=70, in=190] (4,3);
\draw (1,2) to [out=30, in=150] (4,2);
\draw (1,3) to [out=350, in=110] (4,1);
\draw (1,2) to [out=20, in=160] (3,2);
\draw (2,2) to [out=20, in=160] (4,2);
\draw (1,1)--(2,2) (1,2)--(2,1) (3,1)--(4,2) (3,2)--(4,1);
\draw [dashed] (3,2)--(3,3);
\end{tikzpicture}}
&\raisebox{5mm}{$\rightarrow$}&
{\begin{tikzpicture}[scale=0.5]
\foreach \x in {1,2,3,4} {\foreach \y in {1,2,3} {\node at (\x,\y) [fill,circle,inner sep=1pt] {};};};
\draw (1,1)--(1,2) (2,1)--(2,2) (3,1)--(3,2) (4,1)--(4,2)--(4,3); 
\draw (1,1)--(2,1)--(3,1)--(4,1) (1,2)--(2,2)--(3,2)--(4,2) (1,3)--(2,3)--(3,3)--(4,3);
\draw (1,1) to [out=70, in=190] (4,3);
\draw (1,2) to [out=30, in=150] (4,2);
\draw (1,3) to [out=350, in=110] (4,1);
\draw (1,2) to [out=20, in=160] (3,2);
\draw (2,2) to [out=20, in=160] (4,2);
\draw (1,1)--(2,2) (1,2)--(2,1) (3,1)--(4,2) (3,2)--(4,1);
\draw [ultra thick] (4,1) to [out=70, in=290] (4,3);
\end{tikzpicture}}
&\raisebox{5mm}{$\rightarrow$}&
{\begin{tikzpicture}[scale=0.5]
\foreach \x in {1,2,3,4} {\foreach \y in {1,2,3} {\node at (\x,\y) [fill,circle,inner sep=1pt] {};};};
\draw (1,1)--(1,2) (2,1)--(2,2) (3,1)--(3,2) (4,1)--(4,2)--(4,3); 
\draw (1,1)--(2,1)--(3,1)--(4,1) (1,2)--(2,2)--(3,2)--(4,2) (1,3)--(2,3)--(3,3)--(4,3);
\draw (1,1) to [out=70, in=190] (4,3);
\draw (1,2) to [out=30, in=150] (4,2);
\draw (1,2) to [out=20, in=160] (3,2);
\draw (2,2) to [out=20, in=160] (4,2);
\draw (1,1)--(2,2) (1,2)--(2,1) (3,1)--(4,2) (3,2)--(4,1);
\draw (4,1) to [out=70, in=290] (4,3);
\draw [dashed] (1,3) to [out=350, in=110] (4,1);
\end{tikzpicture}}
\\
\raisebox{5mm}{$\rightarrow$}&
{\begin{tikzpicture}[scale=0.5]
\node at (1,3) [fill,circle,inner sep=1pt] {};
\node at (2,3) [fill,circle,inner sep=1pt] {};
\node at (4,3) [fill,circle,inner sep=1pt] {};
\foreach \x in {1,2,3,4} {\foreach \y in {1,2} {\node at (\x,\y) [fill,circle,inner sep=1pt] {};};};
\draw (1,1)--(1,2) (2,1)--(2,2) (3,1)--(3,2) (4,1)--(4,2)--(4,3); 
\draw (1,1)--(2,1)--(3,1)--(4,1) (1,2)--(2,2)--(3,2)--(4,2) (1,3)--(2,3);
\draw (1,1) to [out=70, in=190] (4,3);
\draw (1,2) to [out=30, in=150] (4,2);
\draw (1,2) to [out=20, in=160] (3,2);
\draw (2,2) to [out=20, in=160] (4,2);
\draw (1,1)--(2,2) (1,2)--(2,1) (3,1)--(4,2) (3,2)--(4,1);
\draw (4,1) to [out=70, in=290] (4,3);
\draw [dashed] (2,3)--(4,3);
\node at (3,3) [circle, draw=black, fill=white, inner sep=1pt] {};
\end{tikzpicture}}
&\raisebox{5mm}{$\rightarrow$}&
{\begin{tikzpicture}[scale=0.5]
\node at (1,3) [fill,circle,inner sep=1pt] {};
\node at (2,3) [fill,circle,inner sep=1pt] {};
\node at (4,3) [fill,circle,inner sep=1pt] {};
\foreach \x in {1,2,3,4} {\foreach \y in {1,2} {\node at (\x,\y) [fill,circle,inner sep=1pt] {};};};
\draw (1,1)--(1,2) (2,1)--(2,2) (3,1)--(3,2) (4,1)--(4,2)--(4,3); 
\draw (1,1)--(2,1) (3,1)--(4,1) (1,2)--(2,2)--(3,2)--(4,2) (1,3)--(2,3);
\draw (1,1) to [out=70, in=190] (4,3);
\draw (1,2) to [out=30, in=150] (4,2);
\draw (1,2) to [out=20, in=160] (3,2);
\draw (2,2) to [out=20, in=160] (4,2);
\draw (1,1)--(2,2) (1,2)--(2,1) (3,1)--(4,2) (3,2)--(4,1);
\draw (4,1) to [out=70, in=290] (4,3);
\draw [dashed] (2,1)--(3,1);
\end{tikzpicture}}
&\raisebox{5mm}{$\rightarrow$}&
{\begin{tikzpicture}[scale=0.5]
\node at (1,3) [fill,circle,inner sep=1pt] {};
\node at (2,3) [fill,circle,inner sep=1pt] {};
\node at (4,3) [fill,circle,inner sep=1pt] {};
\foreach \x in {1,2,3,4} {\foreach \y in {1,2} {\node at (\x,\y) [fill,circle,inner sep=1pt] {};};};
\draw (1,1)--(1,2) (2,1)--(2,2) (3,1)--(3,2) (4,1)--(4,2)--(4,3); 
\draw (1,1)--(2,1) (3,1)--(4,1) (1,2)--(2,2)--(3,2)--(4,2) (1,3)--(2,3);
\draw (1,1) to [out=70, in=190] (4,3);
\draw (1,2) to [out=30, in=150] (4,2);
\draw (1,2) to [out=20, in=160] (3,2);
\draw (2,2) to [out=20, in=160] (4,2);
\draw (1,1)--(2,2) (1,2)--(2,1) (3,1)--(4,2) (3,2)--(4,1);
\draw (4,1) to [out=70, in=290] (4,3);
\draw [ultra thick] (3,2)--(4,3);
\end{tikzpicture}}
&\raisebox{5mm}{$\rightarrow$}&
{\begin{tikzpicture}[scale=0.5]
\node at (1,3) [fill,circle,inner sep=1pt] {};
\node at (2,3) [fill,circle,inner sep=1pt] {};
\foreach \x in {1,2,3,4} {\foreach \y in {1,2} {\node at (\x,\y) [fill,circle,inner sep=1pt] {};};};
\draw (1,1)--(1,2) (2,1)--(2,2) (3,1)--(3,2) (4,1)--(4,2); 
\draw (1,1)--(2,1) (3,1)--(4,1) (1,2)--(2,2)--(3,2)--(4,2) (1,3)--(2,3);
\draw (1,2) to [out=30, in=150] (4,2);
\draw (1,2) to [out=20, in=160] (3,2);
\draw (2,2) to [out=20, in=160] (4,2);
\draw (1,1)--(2,2) (1,2)--(2,1) (3,1)--(4,2) (3,2)--(4,1);
\draw [dashed] (1,1) to [out=70, in=190] (4,3);
\draw [dashed] (3,2)--(4,3) (4,2)--(4,3);
\draw [dashed] (4,1) to [out=70, in=290] (4,3);
\node at (4,3) [circle, draw=black, fill=white, inner sep=1pt] {};
\node at (2.5, 1.5) {$G$};
\end{tikzpicture}}
\end{tabular}
\caption{A sequence of operations which induces $I(M_{3,4}) \simpsimeq I(G \sqcup e)$}
\label{M34figure2}
\end{figure}

\end{itemize}
By these base cases, the proof is completed.
\end{proof}

\begin{proof}[Proof of Corollary \ref{CHnsimple}]
Let $M^H_{1,n}$ denote the graph obtained from $M_{2,2n}$ by deleting $n$ edges $1\overline{1}, 3\overline{3}, \ldots, (2n-1)\overline{2n-1}$. We first obtain
 $I(C^H_{1,n+1}) \simpsimeq I(C^H_{1,n+1} \cup (2n+2)\overline{2n+2})$ by
$$\add((2n+2)\overline{3},2), \add((2n+2)\overline{2n+2},\overline{2}), \del((2n+2)\overline{3},2)$$
(see Figure \ref{CH1nsequence}).
\begin{figure}[H]
\centering
\begin{tabular}{cccc}
&
{\begin{tikzpicture}
[scale=0.6]
\foreach \x in {1,2,3,4,5,6} {\foreach \y in {1,2} {\node at (\x,\y) [fill,circle,inner sep=1pt] {};};};
\draw (0.5,1)--(1,1)--(2,1)--(3,1)--(4,1)--(5,1)--(6,1)--(6.5,1);
\draw (0.5,2)--(1,2)--(2,2)--(3,2)--(4,2)--(5,2)--(6,2)--(6.5,2);
\draw (2,1)--(2,2) (4,1)--(4,2) (6,1)--(6,2);
\node at (3,0.5) {\footnotesize $2n+2$};
\node at (4,0.5) {\footnotesize $1$};
\node at (5,0.5) {\footnotesize $2$};
\node at (6,0.5) {\footnotesize $3$};
\end{tikzpicture}}
&\raisebox{7mm}{$\rightarrow$}&
{\begin{tikzpicture}
[scale=0.6]
\foreach \x in {1,2,3,4,5,6} {\foreach \y in {1,2} {\node at (\x,\y) [fill,circle,inner sep=1pt] {};};};
\draw (0.5,1)--(1,1)--(2,1)--(3,1)--(4,1)--(5,1)--(6,1)--(6.5,1);
\draw (0.5,2)--(1,2)--(2,2)--(3,2)--(4,2)--(5,2)--(6,2)--(6.5,2);
\draw (2,1)--(2,2) (4,1)--(4,2) (6,1)--(6,2);
\draw [ultra thick] (3,1)--(6,2);
\node at (1,0.4) {};
\end{tikzpicture}}
\\
\raisebox{2mm}{$\rightarrow$}&
{\begin{tikzpicture}
[scale=0.6]
\foreach \x in {1,2,3,4,5,6} {\foreach \y in {1,2} {\node at (\x,\y) [fill,circle,inner sep=1pt] {};};};
\draw (0.5,1)--(1,1)--(2,1)--(3,1)--(4,1)--(5,1)--(6,1)--(6.5,1);
\draw (0.5,2)--(1,2)--(2,2)--(3,2)--(4,2)--(5,2)--(6,2)--(6.5,2);
\draw (2,1)--(2,2) (4,1)--(4,2) (6,1)--(6,2);
\draw (3,1)--(6,2);
\draw [ultra thick] (3,1)--(3,2);
\end{tikzpicture}}
&\raisebox{2mm}{$\rightarrow$}&
{\begin{tikzpicture}
[scale=0.6]
\foreach \x in {1,2,3,4,5,6} {\foreach \y in {1,2} {\node at (\x,\y) [fill,circle,inner sep=1pt] {};};};
\draw (0.5,1)--(1,1)--(2,1)--(3,1)--(4,1)--(5,1)--(6,1)--(6.5,1);
\draw (0.5,2)--(1,2)--(2,2)--(3,2)--(4,2)--(5,2)--(6,2)--(6.5,2);
\draw (2,1)--(2,2) (3,1)--(3,2) (4,1)--(4,2) (6,1)--(6,2);
\draw [dashed] (3,1)--(6,2);
\end{tikzpicture}}
\end{tabular}
\caption{A sequence of operations which induces $I(C^H_{1,n+1}) \simpsimeq I(C^H_{1,n+1} \cup (2n+2)\overline{2n+2})$}
\label{CH1nsequence}
\end{figure}
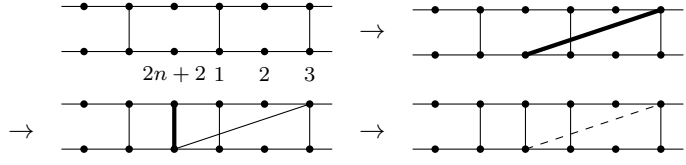
\noindent
Then, as mentioned in Remark \ref{remmainm2}, we can apply Theorem \ref{mainm2} to $C^H_{1, n+1} \cup (2n+2)\overline{2n+2}$ and obtain
$$I(C^H_{1,n+1}) \simpsimeq I(C^H_{1, n+1} \cup (2n+2)\overline{2n+2}) \simpsimeq \Sigma I(M^H_{1,n}).$$
Similarly, we also get
$$I(M^H_{1,n+1}) \simpsimeq \Sigma I(C^H_{1,n}).$$
The base cases are
\begin{align*}
&I(C^H_{1,1}) \searrow *, & &I(M^H_{1,1}) \searrow {\bigvee}_2 {\dia}^0
\end{align*} 
(see Figure \ref{CH11MH11figure}). 
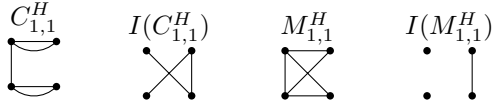
\begin{figure}[H]
\centering
\begin{tabular}{ccccccc}
\begin{tikzpicture}
[scale=0.6]
\foreach \x in {(0,2),(1,2),(0,3),(1,3)} {\node at \x [fill, circle, black, inner sep=1pt] {};};
\draw (1,2)--(0,2)--(0,3)--(1,3);
\draw (0,2) to [out=-45, in=-135] (1,2);
\draw (0,3) to [out=-45, in=-135] (1,3);
\node at (0.5,3.5) {$C^H_{1,1}$};
\end{tikzpicture}
& &
\begin{tikzpicture}
[scale=0.6]
\foreach \x in {(0,2),(1,2),(0,3),(1,3)} {\node at \x [fill, circle, black, inner sep=1pt] {};};
\draw (0,2)--(1,3)--(1,2)--(0,3);
\node at (0.5,3.5) {$I(C^H_{1,1})$};
\end{tikzpicture}
& &
\begin{tikzpicture}
[scale=0.6]
\foreach \x in {(0,2),(1,2),(0,3),(1,3)} {\node at \x [fill, circle, black, inner sep=1pt] {};};
\draw (1,2)--(0,2)--(0,3)--(1,3);
\draw (0,2)--(1,3) (0,3)--(1,2);
\node at (0.5,3.5) {$M^H_{1,1}$};
\end{tikzpicture}
& &
\begin{tikzpicture}
[scale=0.6]
\foreach \x in {(0,2),(1,2),(0,3),(1,3)} {\node at \x [fill, circle, black, inner sep=1pt] {};};
\draw (1,2)--(1,3);
\node at (0.5,3.5) {$I(M^H_{1,1})$};
\end{tikzpicture}
\end{tabular}
\caption{$C^H_{1,1}$, $M^H_{1,1}$ and their independence complexes}
\label{CH11MH11figure}
\end{figure}
By induction, we complete the proof.
\end{proof}

\appendix
\section{Computation of $\widetilde{\chi}(I(P_{4,n}))$}
\label{appendixA}
Recall that the reduced Euler characteristic $\widetilde{\chi} (K)$ of a simplicial complex $K$ is defined by 
$$\widetilde{\chi}(K) = \sum_{\sigma \in K} (-1)^{\dim \sigma}.$$
\begin{proposition}
\label{EulerP4n}
We have
\begin{align*}
&\widetilde{\chi}(I(P_{4,6k+i})) = \left \{
\begin{aligned}
&-2k-1 &\quad &(i = 0, 2), \\
&2k &\quad &(i = 1), \\
&2k+1 &\quad &(i = 3,5), \\
&-2k-2 &\quad &(i=4) .
\end{aligned} \right. 
\end{align*}
\end{proposition}

We use the following lemma, which was shown by Adamaszek \cite{Ada12}.
\begin{lemma}{\cite[Proposition 3.4]{Ada12}}
\label{delcofib}
Let $G$ be a graph and $e$ be an edge of $G$. Then we have the following cofiber sequence:
$$\begin{tikzcd}
I((G \setminus N[e]) \sqcup e) \ar[r] & I(G) \ar[r] &I(G-e), 
\end{tikzcd}$$
where the two maps are the inclusion maps.
\end{lemma}

\begin{proof}[Proof of Proposition \ref{EulerP4n}]
Let $X_n = P_{4,n} \cup 1\widetilde{1}$ and $Y_n = P_{4,n} \setminus \{1, \widetilde{1} \}$. We obtain $I(P_{4,n}) \simpsimeq \Sigma^2 I(X_{n-2})$ by
\begin{align*}
&\del(\overline{2},1), \del(\widehat{2},\widetilde{1}), \add(\widehat{1}3,1), \del(23,\overline{1}), \del(\overline{1},2), 
\add(3\widetilde{3},\widetilde{1}), \del(\widetilde{2}\widetilde{3},\widehat{1}), \del(\widehat{1},\widetilde{2})
\end{align*}
(see Figure \ref{P4nfigure}). Then it follows that
\begin{align}
\label{P4nX}
\widetilde{\chi}(I(P_{4,n})) =\widetilde{\chi}(I(X_{n-2})). 
\end{align}
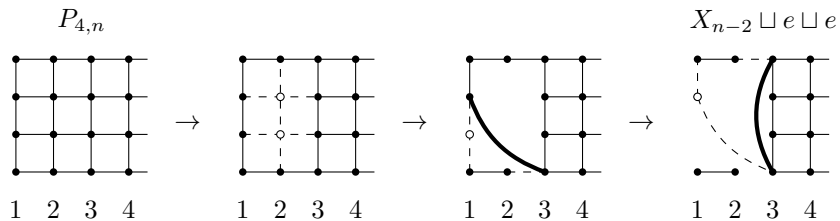
\begin{figure}[H]
\centering
\begin{tabular}{ccccccc}
{\begin{tikzpicture}
[scale=0.5]
\draw (1,1)--(4.5,1) (1,2)--(4.5,2) (1,3)--(4.5,3) (1,4)--(4.5,4);
\draw (1,1)--(1,4) (2,1)--(2,4) (3,1)--(3,4) (4,1)--(4,4);
\foreach \x in {1,2,3,4} {\node at (\x,0) {\x};};
\foreach \x in {1,2,3,4} {\node at (\x, 1) [fill,circle,inner sep=1pt] {};};
\foreach \x in {1,2,3,4} {\node at (\x, 2) [fill,circle,inner sep=1pt] {};};
\foreach \x in {1,2,3,4} {\node at (\x, 3) [fill,circle,inner sep=1pt] {};};
\foreach \x in {1,2,3,4} {\node at (\x, 4) [fill,circle,inner sep=1pt] {};};
\node at (2.75,5) {$P_{4,n}$};
\end{tikzpicture}}
&\raisebox{13mm}{$\rightarrow$}&
{\begin{tikzpicture}
[scale=0.5]
\draw (1,1)--(4.5,1) (3,2)--(4.5,2) (3,3)--(4.5,3) (1,4)--(4.5,4);
\draw (1,1)--(1,4)  (3,1)--(3,4) (4,1)--(4,4);
\draw [dashed] (2,1)--(2,4) (1,2)--(3,2) (1,3)--(3,3);
\foreach \x in {1,2,3,4} {\node at (\x,0) {\x};};
\foreach \x in {1,2,3,4} {\node at (\x, 1) [fill,circle,inner sep=1pt] {};};
\foreach \x in {1,3,4} {\node at (\x, 2) [fill,circle,inner sep=1pt] {};};
\foreach \x in {1,3,4} {\node at (\x, 3) [fill,circle,inner sep=1pt] {};};
\foreach \x in {1,2,3,4} {\node at (\x, 4) [fill,circle,inner sep=1pt] {};};
\node at (2,2) [circle, draw=black, fill=white, inner sep=1pt] {};
\node at (2,3) [circle, draw=black, fill=white, inner sep=1pt] {};
\node at (2.25,5) {};
\end{tikzpicture}}
&\raisebox{13mm}{$\rightarrow$}&
{\begin{tikzpicture}
[scale=0.5]
\draw (1,1)--(2,1) (3,1)--(4.5,1) (3,2)--(4.5,2) (3,3)--(4.5,3) (1,4)--(4.5,4);
\draw (1,3)--(1,4)  (3,1)--(3,4) (4,1)--(4,4);
\draw [dashed] (2,1)--(3,1) (1,1)--(1,3);
\foreach \x in {1,2,3,4} {\node at (\x,0) {\x};};
\foreach \x in {1,2,3,4} {\node at (\x, 1) [fill,circle,inner sep=1pt] {};};
\foreach \x in {3,4} {\node at (\x, 2) [fill,circle,inner sep=1pt] {};};
\foreach \x in {1,3,4} {\node at (\x, 3) [fill,circle,inner sep=1pt] {};};
\foreach \x in {1,2,3,4} {\node at (\x, 4) [fill,circle,inner sep=1pt] {};};
\draw [ultra thick] (1,3) to [out=290, in=160] (3,1);
\node at (1,2) [circle, draw=black, fill=white, inner sep=1pt] {};
\node at (2.25,5) {};
\end{tikzpicture}}
&\raisebox{13mm}{$\rightarrow$}&
{\begin{tikzpicture}
[scale=0.5]
\draw (1,1)--(2,1) (3,1)--(4.5,1) (1,4)--(2,4) (3,2)--(4.5,2) (3,3)--(4.5,3) (3,4)--(4.5,4);
\draw (3,1)--(3,4) (4,1)--(4,4);
\draw [dashed] (2,4)--(3,4) (1,3)--(1,4);
\foreach \x in {1,2,3,4} {\node at (\x,0) {\x};};
\foreach \x in {1,2,3,4} {\node at (\x, 1) [fill,circle,inner sep=1pt] {};};
\foreach \x in {3,4} {\node at (\x, 2) [fill,circle,inner sep=1pt] {};};
\foreach \x in {1,3,4} {\node at (\x, 3) [fill,circle,inner sep=1pt] {};};
\foreach \x in {1,2,3,4} {\node at (\x, 4) [fill,circle,inner sep=1pt] {};};
\draw [dashed] (1,3) to [out=290, in=160] (3,1);
\draw [ultra thick] (3,1) to [out=120, in=240] (3,4);
\node at (1,3) [circle, draw=black, fill=white, inner sep=1pt] {};
\node at (2.75,5) {$X_{n-2} \sqcup e \sqcup e$};
\end{tikzpicture}}
\end{tabular}
\caption{A sequence of operations which induces $I(P_{4,n}) \simpsimeq I(X_{n-2} \sqcup e \sqcup e)$}
\label{P4nfigure}
\end{figure}

By Lemma \ref{delcofib}, a removal of the edge $1\widetilde{1}$ from $X_{n-2}$ yields the following cofiber sequence:
$$\begin{tikzcd}
\Sigma I(Y_{n-3}) \ar[r] & I(X_{n-2}) \ar[r] & I(P_{4,n-2}), 
\end{tikzcd}$$
which implies
\begin{align}
\label{P4ncofiber}
\widetilde{\chi}(I(P_{4,n-2})) &= \widetilde{\chi}(I(X_{n-2})) - \widetilde{\chi}(\Sigma I(Y_{n-3})) \nonumber \\
&= \widetilde{\chi}(I(X_{n-2})) + \widetilde{\chi}(I(Y_{n-3})) .
\end{align}
By (\ref{P4nX}) and (\ref{P4ncofiber}), we have
\begin{align}
\label{P4ninduction}
\widetilde{\chi}(I(P_{4,n})) =\widetilde{\chi}(I(P_{4,n-2})) - \widetilde{\chi}(I(Y_{n-3})) .
\end{align}
As a base case, we have $I(P_{4,1}) \simpsimeq *$ and $I(P_{4,2}) \simpsimeq {\dia}^1$, which imply
\begin{align}
\label{P4nbasecase}
\widetilde{\chi}(I(P_{4,1})) = 0, & & \widetilde{\chi}(I(P_{4,2})) =-1.
\end{align}
Therefore, it is sufficient to compute $\widetilde{\chi}(I(Y_n))$.

We obtain $I(Y_n) \simpsimeq \Sigma^3 I(Y_{n-3})$ by
\begin{align*}
\del(\overline{2},\widehat{1}), \del(\widehat{2},\overline{1}), \del(\overline{3},2), \del(\widehat{3},\widetilde{2}), 
\del(4,2), \del(\widetilde{4},\widetilde{2})
\end{align*}
(see Figure \ref{Ynfigure}).
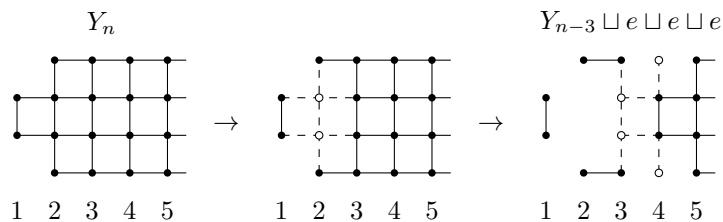
\begin{figure}[H]
\centering
\begin{tabular}{ccccc}
{\begin{tikzpicture}
[scale=0.5]
\draw (2,1)--(5.5,1) (1,2)--(5.5,2) (1,3)--(5.5,3) (2,4)--(5.5,4);
\draw (1,2)--(1,3) (2,1)--(2,4) (3,1)--(3,4) (4,1)--(4,4) (5,1)--(5,4);
\foreach \x in {1,2,3,4,5} {\node at (\x,0) {\x};};
\foreach \x in {2,3,4,5} {\node at (\x, 1) [fill,circle,inner sep=1pt] {};};
\foreach \x in {1,2,3,4,5} {\node at (\x, 2) [fill,circle,inner sep=1pt] {};};
\foreach \x in {1,2,3,4,5} {\node at (\x, 3) [fill,circle,inner sep=1pt] {};};
\foreach \x in {2,3,4,5} {\node at (\x, 4) [fill,circle,inner sep=1pt] {};};
\node at (3.25,5) {$Y_n$};
\end{tikzpicture}}
&\raisebox{13mm}{$\rightarrow$}&
{\begin{tikzpicture}
[scale=0.5]
\draw (2,1)--(5.5,1) (3,2)--(5.5,2) (3,3)--(5.5,3) (2,4)--(5.5,4);
\draw (1,2)--(1,3) (3,1)--(3,4) (4,1)--(4,4) (5,1)--(5,4);
\foreach \x in {1,2,3,4,5} {\node at (\x,0) {\x};};
\foreach \x in {2,3,4,5} {\node at (\x, 1) [fill,circle,inner sep=1pt] {};};
\foreach \x in {1,3,4,5} {\node at (\x, 2) [fill,circle,inner sep=1pt] {};};
\foreach \x in {1,3,4,5} {\node at (\x, 3) [fill,circle,inner sep=1pt] {};};
\foreach \x in {2,3,4,5} {\node at (\x, 4) [fill,circle,inner sep=1pt] {};};
\draw [dashed] (2,1)--(2,4) (1,2)--(3,2) (1,3)--(3,3);
\node at (2,2) [circle, draw=black, fill=white, inner sep=1pt] {};
\node at (2,3) [circle, draw=black, fill=white, inner sep=1pt] {};
\node at (3.25,5) {};
\end{tikzpicture}}
&\raisebox{13mm}{$\rightarrow$}&
{\begin{tikzpicture}
[scale=0.5]
\draw (2,1)--(3,1) (5,1)--(5.5,1) (4,2)--(5.5,2) (4,3)--(5.5,3) (2,4)--(3,4) (5,4)--(5.5,4);
\draw (1,2)--(1,3) (4,2)--(4,3) (5,1)--(5,4);
\foreach \x in {1,2,3,4,5} {\node at (\x,0) {\x};};
\foreach \x in {2,3,5} {\node at (\x, 1) [fill,circle,inner sep=1pt] {};};
\foreach \x in {1,4,5} {\node at (\x, 2) [fill,circle,inner sep=1pt] {};};
\foreach \x in {1,4,5} {\node at (\x, 3) [fill,circle,inner sep=1pt] {};};
\foreach \x in {2,3,5} {\node at (\x, 4) [fill,circle,inner sep=1pt] {};};
\draw [dashed] (3,1)--(3,4) (3,2)--(4,2) (3,3)--(4,3) (4,1)--(4,2) (4,3)--(4,4);
\node at (3,2) [circle, draw=black, fill=white, inner sep=1pt] {};
\node at (3,3) [circle, draw=black, fill=white, inner sep=1pt] {};
\node at (4,1) [circle, draw=black, fill=white, inner sep=1pt] {};
\node at (4,4) [circle, draw=black, fill=white, inner sep=1pt] {};
\node at (3.25,5) {$Y_{n-3} \sqcup e \sqcup e \sqcup e$};
\end{tikzpicture}}
\end{tabular}
\caption{A sequence of operations which induces $I(Y_n) \simpsimeq I(Y_{n-3} \sqcup e \sqcup e \sqcup e)$}
\label{Ynfigure}
\end{figure}

This implies
\begin{align}
\label{P4nrecursion}
\widetilde{\chi}(I(Y_{n})) = - \widetilde{\chi}(I(Y_{n-3})).
\end{align}
As a base case, we have $I(Y_1) \simpsimeq {\dia}^0$, $I(Y_2) \simpsimeq *$ and $I(Y_3) \simpsimeq {\dia}^2$ (see Figure \ref{P4nYnbasecase}). 
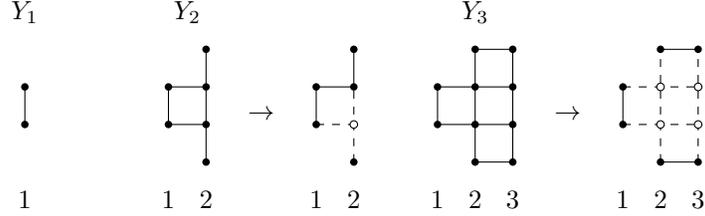
\begin{figure}[H]
\centering
\begin{tabular}{cccccccccccc}
{\begin{tikzpicture}
[scale=0.5]
\draw (1,2)--(1,3);
\foreach \x in {1} {\node at (\x, 0) {\x};};
\foreach \x in {1} {\node at (\x, 2) [fill,circle,inner sep=1pt] {};};
\foreach \x in {1} {\node at (\x, 3) [fill,circle,inner sep=1pt] {};};
\node at (1,5) {$Y_1$};
\end{tikzpicture}}
& & & &
{\begin{tikzpicture}
[scale=0.5]
\draw (1,2)--(1,3) (2,1)--(2,4);
\draw (1,2)--(2,2) (1,3)--(2,3);
\foreach \x in {1,2} {\node at (\x, 0) {\x};};
\foreach \x in {2} {\node at (\x, 1) [fill,circle,inner sep=1pt] {};};
\foreach \x in {1,2} {\node at (\x, 2) [fill,circle,inner sep=1pt] {};};
\foreach \x in {1,2} {\node at (\x, 3) [fill,circle,inner sep=1pt] {};};
\foreach \x in {2} {\node at (\x, 4) [fill,circle,inner sep=1pt] {};};
\node at (1.5,5) {$Y_2$};
\end{tikzpicture}}
&\raisebox{13mm}{$\rightarrow$}&
{\begin{tikzpicture}
[scale=0.5]
\draw (1,2)--(1,3) (2,3)--(2,4);
\draw (1,3)--(2,3);
\draw [dashed] (1,2)--(2,2) (2,1)--(2,3);
\foreach \x in {1,2} {\node at (\x, 0) {\x};};
\foreach \x in {2} {\node at (\x, 1) [fill,circle,inner sep=1pt] {};};
\foreach \x in {1} {\node at (\x, 2) [fill,circle,inner sep=1pt] {};};
\foreach \x in {1,2} {\node at (\x, 3) [fill,circle,inner sep=1pt] {};};
\foreach \x in {2} {\node at (\x, 4) [fill,circle,inner sep=1pt] {};};
\node at (2,2) [circle, draw=black, fill=white, inner sep=1pt] {};
\node at (1.5,5) {};
\end{tikzpicture}}
& &
{\begin{tikzpicture}
[scale=0.5]
\draw (1,2)--(1,3) (2,1)--(2,4) (3,1)--(3,4);
\draw (1,2)--(3,2) (1,3)--(3,3) (2,1)--(3,1) (2,4)--(3,4);
\foreach \x in {1,2,3} {\node at (\x, 0) {\x};};
\foreach \x in {2,3} {\node at (\x, 1) [fill,circle,inner sep=1pt] {};};
\foreach \x in {1,2,3} {\node at (\x, 2) [fill,circle,inner sep=1pt] {};};
\foreach \x in {1,2,3} {\node at (\x, 3) [fill,circle,inner sep=1pt] {};};
\foreach \x in {2,3} {\node at (\x, 4) [fill,circle,inner sep=1pt] {};};
\node at (2,5) {$Y_3$};
\end{tikzpicture}}
&\raisebox{13mm}{$\rightarrow$}&
{\begin{tikzpicture}
[scale=0.5]
\draw (1,2)--(1,3) (2,1)--(3,1) (2,4)--(3,4);
\draw [dashed] (2,1)--(2,4) (3,1)--(3,4) (1,2)--(3,2) (1,3)--(3,3) ;
\foreach \x in {1,2,3} {\node at (\x, 0) {\x};};
\foreach \x in {2,3} {\node at (\x, 1) [fill,circle,inner sep=1pt] {};};
\foreach \x in {1} {\node at (\x, 2) [fill,circle,inner sep=1pt] {};};
\foreach \x in {1} {\node at (\x, 3) [fill,circle,inner sep=1pt] {};};
\foreach \x in {2,3} {\node at (\x, 4) [fill,circle,inner sep=1pt] {};};
\foreach \x in {2,3} {\node at (\x, 2) [circle, draw=black, fill=white, inner sep=1pt] {};};
\foreach \x in {2,3} {\node at (\x, 3) [circle, draw=black, fill=white, inner sep=1pt] {};};
\node at (2,5) {};
\end{tikzpicture}}
\end{tabular}
\caption{$Y_1$, $Y_2$, $Y_3$ and the operations}
\label{P4nYnbasecase}
\end{figure}

Therefore, by (\ref{P4nrecursion}), we have 
\begin{align}
&\widetilde{\chi}(I(Y_{6k+i})) = \left \{
\begin{aligned}
\label{Yn}
&-1 &\quad &(i = 0, 4), \\
&1 &\quad &(i = 1,3), \\
&0 &\quad &(i = 2,5) .
\end{aligned} \right. 
\end{align}
Then, by (\ref{P4ninduction}), (\ref{P4nbasecase}) and (\ref{Yn}), we obtain the desired conclusion.
\end{proof}


\begin{thebibliography}{1}

\bibitem{Adahard}
Micha{\l} Adamaszek.
\newblock {Hard squares on cylinders revisited}.
\newblock {\em arXiv e-prints}, arXiv:1202.1655, Feb 2012.

\bibitem{Ada12}
Micha{\l} Adamaszek.
\newblock Splittings of independence complexes and the powers of cycles.
\newblock {\em Journal of Combinatorial Theory, Series A}, 119:1031--1047,
  2012.

\bibitem{Csorba09}
P{\'e}ter Csorba.
\newblock Subdivision yields Alexander duality on independence complexes.
\newblock {\em The Electronic Journal of Combinatorics}, 16(2):\#R11, 2009.

\bibitem{Engs09}
Alexander Engstr{\"{o}}m.
\newblock Complexes of directed trees and independence complexes.
\newblock {\em Discrete Mathematics}, 309:3299--3309, 2009.

\bibitem{Iriye12}
Kouyemon Iriye.
\newblock On the homotopy types of the independence complexes of grid graphs
  with cylindrical identification.
\newblock {\em Kyoto Journal of Mathematics}, 52(3):479--501, 2012.

\bibitem{Koz99}
Dmitry~N Kozlov.
\newblock Complexes of directed trees.
\newblock {\em Journal of Combinatorial theory, Series A}, 88:112--122, 1999.

\bibitem{Thap08}
Johan {Thapper}.
\newblock {Independence Complexes of Cylinders Constructed from Square and
  Hexagonal Grid Graphs}.
\newblock {\em arXiv e-prints}, arXiv:0812.1165, Dec 2008.

\end{thebibliography}
\end{document}